\documentclass[12pt]{article}
\usepackage{graphicx} 
\usepackage{amsmath}
\usepackage{amsthm}
\usepackage{amssymb}
\usepackage{enumitem}
\usepackage{geometry}
\usepackage{mathrsfs}
\usepackage{lmodern}
\usepackage{upgreek}
\usepackage[
backend=biber,
sorting=nyt
]{biblatex}
\usepackage[colorlinks=true,
linkcolor=Peach,
citecolor=Peach]{hyperref}
\usepackage[dvipsnames, svgnames]{xcolor}
\hypersetup{
colorlinks=true,
linktoc=all,
linkcolor=RoyalBlue
}
\usepackage{csquotes}
\usepackage[T1]{fontenc} 
\usepackage{dsfont}
\title{Fluctuations of the planar random walk sausage}
\author{Maxence Baccara \thanks{Ecole Polytechnique, Palaiseau, France. maxence.baccara@polytechnique.edu} \and  Joelle Ishak \thanks{Université Sorbonne Paris Nord, Villetaneuse, France. ishak@math.univ-paris13.fr}}

\newcommand{\ZZ}{\mathbb{Z}}
\newcommand{\PP}{\mathbb{P}}
\newcommand{\EE}{\mathbb{E}}
\newcommand{\RR}{\mathbb{R}}
\newcommand{\R}{\mathcal{R}}
\newcommand{\I}{\mathcal{I}}
\newcommand{\J}{\mathcal{J}}
\newcommand{\TT}{\mathbb{T}}
\newcommand{\CC}{\mathbb C}
\renewcommand{\Cap}{\mathrm{Cap}}
\renewcommand{\d}{\mathrm{d}}
\newcommand{\1}{\mathds{1}}
\renewcommand{\i}{\mathrm{i}}

\newtheorem{lemma}{Lemma}
\newtheorem{proposition}{Proposition}
\newtheorem{theorem}{Theorem}
\newtheorem{corollary}{Corollary}

\numberwithin{equation}{section}

\begin{document}

\maketitle

\begin{abstract}
We establish fluctuation results for the volume of  random walk sausage $\R_n^{A}$ associated to a finite set $A\subset\ZZ^2$, where the underlying random walk $S$ has finite second moments. Precisely, we establish that the fluctuations are non-Gaussian and are described by the renormalized self intersection local time of planar Brownian motion, as in the cases of the range of random walks and the Wiener sausage. For any $N\ge1$, we also establish an asymptotic expansion of the expected volume of the form
\begin{equation*}
    \EE[\R_n^{A}]=n\sum_{k=1}^N c(k,A,S)/(\log n)^k+O\left(n/(\log n)^{N+1}\right),
\end{equation*}
where the constants $c(k,A,S)$ depend only on the order $k$, the set $A$ and the distribution of $S$, and can be computed explicitly in terms of complete exponential Bell polynomials.
\end{abstract}

\section{Introduction}

Let $S:=(S_n)_{n\ge0}$ be a random walk in $\ZZ^2$. The range $R_n$ of $S$ is the set of sites visited by $S$ up to time $n$, $i.e.$ $R_n:=\{S_0,\dots,S_n\}$. The object of interest in this paper is the volume of the random walk sausage associated to a finite and nonempty subset $A \subset \mathbb{Z}^2$, which we write
\begin{equation*}
    \R_n^{A}:=|R_n+A|.
\end{equation*}
 
Since $S$ is invariant by translation, we can (and systematically will) suppose without loss of generality that $0\in A$. Also, we shall suppose that $A$ is symmetric about $0$, that is 
$A=-A$. We make this additional assumption on $A$ only to simplify some notation, but it is clear that the proofs presented below still work without this assumption. When $A=\{0\}$, the quantity $\R_n^A$ reduces to the volume of the range. 
In dimension $d\geq3$, for the volume of the range of a transient random walk, Jain and Pruitt \cite{Jain1974} established central limit theorems with Gaussian limits. Moreover, in a celebrated paper of Le Gall \cite{LeGall1986}, it is shown that for $d=2$ and under some second moment hypothesis,
\begin{equation*}
         \frac{(\log n)^2}{n}\left(\R_n^{\{0\}}-\EE\left[\R_n^{\{0\}}\right]\right)\underset{n\rightarrow\infty}\implies -4\pi^2\sigma^2\gamma_1,
\end{equation*}
where the covariance matrix of the random walk is $\sigma^2\mathrm{Id}$ for some $\sigma^2>0$. The random variable $\gamma_1$ is formally defined by
\begin{equation*}
    \gamma_1:=\int_{0<s<t<1}\d s\d t\left(\delta_0(W_t-W_s)-\EE\left[\delta_0(W_t-W_s)\right]\right),
\end{equation*}
where $W$ is a standard Brownian motion in $\RR^2$. The variable $\gamma_1$ is referred to as the renormalized self-intersection local time of standard Brownian motion. For more information on the rigorous construction of the renormalized self-intersection time, one can consult \cite{LeGall1985}, \cite{LeGallStFlour}, \cite{Varadhan1969}.

The fluctuation result was extended to random walks in $\ZZ^d$ in the domain of attraction of $\alpha$-stable laws by Le Gall and Rosen \cite{LeGallRosen1991}. Their approach relies on a dyadic decomposition of the trajectory and on a comparison between intersections of two independent ranges and the mutual intersection local time of two independent copies of the limiting stable process. In the regime $d/\alpha\in[1,3/2)$, they establish that the fluctuations of the range are described by the renormalized self intersection time of the limiting $\alpha$-stable Lévy process.

Heuristically speaking, the range of a random walk is the discrete analogue of the Wiener sausage (see \cite{LeGall1986}, Section $6$ for a discussion on this matter). Given a compact subset $K\subset\RR^2$, the Wiener sausage $W^k(0,t)$ associated to $K$ up to time $t$ is the Minkowski sum of the set $K$ and $\{W_s\mid s\in[0,t]\}$, namely
\begin{equation*}
    W^K(0,t):=\{W_s+x\mid s\in[0,t], x\in K\},
\end{equation*}
and is geometrically obtained by letting the compact $K$ slide along the trajectory of $W$ up to time $t$. In dimensions $d\geq3$, the classical theorem of Kesten, Spitzer and Whitman states that 
\begin{equation*}
    t^{-1}|W^K(0,t)|\xrightarrow[t\to\infty]{\mathrm{a.s.}}\Cap(K),
\end{equation*}
where $\Cap(K)$ is the Newtonian capacity of $K$ (see \cite{ItoMcKean1974}, p. 252). The planar setting is fundamentally different. In dimension two, Newtonian potential theory is replaced by logarithmic potential theory, and the corresponding potential-theoretic quantity is the logarithmic capacity of $K$. Spitzer \cite{Spitzer1991} studied the expected volume of the Wiener sausage in dimension $2$ and $3$. The result in the plane reads
\begin{equation*}
    \EE\left[|W^K(0,t)|\right]=2\pi\frac{t}{\log t}+2\pi\left(1+C_E-\log 2+R(K)\right)\frac{t}{(\log t)^2}+o\left(\frac{t}{(\log t)^2}\right),
\end{equation*}
where $C_E$ is the Euler-Mascheroni constant and where $R(K)$ is the logarithm of the logarithmic capacity of $K$. The fluctuations of the Wiener sausage were subsequently investigated by Le Gall \cite{LeGallSausage1988}. In dimension two, the fluctuation result is  
\begin{equation}
    \label{limit:CLT for wiener sausage}
    \frac{(\log t)^2}{t}\left(|W^K(0,t)|-2\pi\frac{t}{\log t}\right)\underset{t\rightarrow\infty}\implies 2\pi\left(1+C_E-\log 2+R(K)-2\pi\gamma_1\right).
\end{equation}
Rosen \cite{RosenStableSausage1992} later extended the asymptotic analysis to planar sausages associated with symmetric stable processes. Moreover, Le Gall \cite{LeGallSausage1988} studied the fluctuations in dimensions $d \geq 3$ and established a Gaussian limit.

The Wiener sausage has a natural discrete counterpart in $\R_n^{A}$. For the simple random walk in dimensions $d\geq3$, the discrete analogue of the Kesten, Spitzer and Whitman theorem \cite{Spitzer1973Discussion} states that, for every finite set $A\subset\mathbb Z^d$,
\begin{equation*}
    n^{-1}\R_n^A\xrightarrow[n\to\infty]{\mathrm{a.s.}} \mathrm{Cap}(A),
\end{equation*}
where $\mathrm{Cap}(A)$ capacity of $A$ in $\ZZ^d$. 

As for the continuous setting, the situation is fundamentally different in dimension $2$. The appropriate notion of capacity is different and is defined in terms of the potential kernel and harmonic measure from infinity (see Section \ref{development of the sausage}). In this recurrent setting, as in the continuous case, the leading term of the expected random walk sausage volume is independent of the set $A$, whereas its capacity appears in the second-order term. More precisely, Port \cite{Port} proved that, for finite sets $A,B\subset\mathbb Z^2$,
\begin{equation*}
    \mathbb E\left[\mathcal R_n^A\right]
-
\mathbb E\left[\mathcal R_n^B\right]
\sim
4\pi^2\sigma^4
\bigl(\mathrm{Cap}(A)-\mathrm{Cap}(B)\bigr)
\frac{n}{(\log n)^2}.
\end{equation*}
Combined with Le Gall and Rosen's second-order expansion for the expected range (\cite{LeGallRosen1991}, Theorem $6.9$), this yields a two-term asymptotic expansion for $\mathbb E[\R_n^A]$. These results naturally raise the question of whether planar random walk sausages have the same limiting fluctuations as the range and as their continuous counterparts.

\subsection{Main results}
Before presenting the main results of the paper, we will state the exact assumptions under which we will be working, namely
\begin{enumerate}[label=(A\arabic*)]
    \item
    \label{Assumption A1}
    The walk $S$ is centered, symmetric has finite second moments and has covariance matrix $\sigma^2\mathrm{Id}$ for some $\sigma^2>0$.
    \item
    \label{Assumption A2}
    $S$ is strongly aperiodic, $i.e.$ the subgroup generated by $\{x\in\ZZ^2\mid \PP(S_1=x)>0\}$ is $\ZZ^2$ itself.
    \item
    \label{Assumption A3}
    For some $\delta>0$, we have  $\EE[|S_1|^{2+\delta}]<\infty$.
\end{enumerate}

Assumption \ref{Assumption A1} is crucial, since it tells us that once conveniently rescaled, the trajectories of $S$ converge to standard Brownian motion in $\RR^2$. This allows us to give a characterization of the fluctuations. Assumption \ref{Assumption A2} is used to simplify the statement of the local limit Theorem, which we use to obtain estimates for the Green function. It is, however, not crucial, as any non aperiodic random walk $S$ can be made aperiodic by considering the associated lazy random walk $\widetilde S$. Finally, Assumption \ref{Assumption A3} is used to obtain sharper Green's function estimates, which lead to the asymptotic expansion of the expected volume of the sausage. It would be interesting to see if these results still hold if we relax this hypothesis, in particular understanding how our results could be generalized to heavy tailed random walks.

We are now ready to state the main results of the paper. Firstly, we prove an asymptotic expansion of the expected volume of the sausage.

\begin{theorem}
\label{theorem:expansion of the mean}
    Let $A$ be a finite subset of $\ZZ^2$ containing $0$. Then under assumptions \ref{Assumption A1}, \ref{Assumption A2}, \ref{Assumption A3}, we have for any $N\ge1$,
    \begin{equation*}
        \EE\left[\R_n^{A}\right]=2\pi\sigma^2\frac{n}{\log n}+\sum_{k=2}^Nc(k,A,S)\frac{n}{(\log n)^k}+O\left(\frac{n}{(\log n)^{N+1}}\right).
    \end{equation*}
    The constants $c(k,A,S)$ depend only on the rank $k$, the set $A$ and the distribution of $S$, and can be expressed as
    \begin{equation*}
        c(k,A,S)=2\pi\sigma^2 B_{k-1}\left(\rho_{A,S},1-\zeta(2),2(1-\zeta(3)),\dots,(k-2)!(1-\zeta(k-1))\right),
    \end{equation*}
    where $B_{k-1}$ is the $(k-1)$-th complete exponential Bell polynomial and where $\rho_{A,S}:=1-C_E+2\pi\sigma^2(\Cap(A)-\kappa_S)$, for $C_E$ denoting the Euler-Mascheroni constant, $\mathrm{Cap}(A)$ the capacity of $A$ (see Section \ref{development of the sausage} for the definition) and $\kappa_S$ a constant depending only on the random walk $S$ (see Appendix \ref{appendix:kappa_S}).
\end{theorem}
At the cost of adding the additional moment assumption \ref{Assumption A3}, we improve upon the results of Port \cite{Port} and Le Gall, Rosen \cite{LeGallRosen1991}, which are limited to the case $N=2$. Our approach differs greatly from \cite{LeGallRosen1991}, where the authors perform calculations directly on the range of the random walk at a deterministic time $n$, whereas we rather use geometric killing and then invoke a powerful Tauberian Theorem due to Flajolet and Odlyzko \cite{FlajoletOdlyzko}, which is ultimately what allows us to obtain our expansion at an arbitrary order. Some basic facts regarding complete exponential Bell Polynomials are recalled in Appendix \ref{appendix:bell polynomials}, where we also give the first few values of $c(k,A,S)$.
Our next main result concerns the fluctuations of the volume of the random walk sausage.
\begin{theorem}
\label{theorem:CLT}
     Let $A$ be a finite subset of $\ZZ^2$ containing $0$. Under assumptions \ref{Assumption A1} and \ref{Assumption A2}, we have
     \begin{equation*}
         \frac{(\log n)^2}{n}\left(\R_n^{A}-\EE\left[\R_n^{A}\right]\right)\underset{n\rightarrow\infty}\implies -4\pi^2\sigma^2\gamma_1,
     \end{equation*}
     where $\gamma_1$ is the renormalized intersection local time of standard Brownian motion in $\RR^2$.
\end{theorem}
Interestingly, the fluctuations don't depend on the set $A$. However, combining Theorem~\ref{theorem:CLT} and Theorem~\ref{theorem:expansion of the mean} in the case $N=2$, we immediately get the following
\begin{corollary}
    Under assumptions of Theorem \ref{theorem:expansion of the mean}, we have 
    \begin{equation}
    \label{limit:CLT 2 random walk sausage}
        \frac{(\log n)^2}{n}\left(\R_n^{A}-\frac{2\pi\sigma^2n}{\log n}\right)\underset{n\rightarrow\infty}\implies2\pi\sigma^2\left(1-C_E+2\pi\sigma^2(\mathrm{Cap}(A)-\kappa_S)-2\pi\gamma_1\right).
    \end{equation}
\end{corollary}

\subsection{Organization of the paper and strategy of proof}

The paper is organized as follows. In Section~\ref{development of the sausage}, we introduce basic notions of potential theory in $\ZZ^2$, which are largely different from the notions used in dimensions $d\ge3$ due to the recurrence of the random walk. We then give a proof of Theorem~\ref{theorem:expansion of the mean}. The idea of the proof is to kill the random walk at an independent geometric time $\xi_z$ with mean $z(1-z)^{-1}$ and to obtain an asymptotic expansion as $z\rightarrow 1$ of $\EE\left[\R_{\xi_z}^{A}\right]$. This quantity is calculated in two different ways, namely we have the key identities 
\begin{equation*}
    \EE\left[\R_{\xi_z}^{A}\right]=\sum_{x\in\ZZ^2}\EE_x\left[z^{T_A}\right]=(1-z)\sum_{n\ge0}z^n\EE[\R_n^{A}].
\end{equation*}
This allows us to get a closed expression for the power series of $\EE[\R_n^{A}]$ of the form (cf \eqref{eq:power series expression for E[R_n^A]})
\begin{equation*}
    \sum_{n\ge0}z^n\EE\left[\R_n^{A}\right]=\frac{(1-z)^{-2}}{\log(1/(1-z))}\left(\sum_{k=0}^{N-1}\frac{C_k}{\log(1/(1-z))^k}+O\left(\frac{1}{\log(1/(1-z))^N}\right)\right),
\end{equation*}
for arbitrary $N\ge1$ and explicit constants $C_k$. This precise form allows us to invoke a Tauberian Theorem of Flajolet and Odlyzko \cite{FlajoletOdlyzko} to recover the desired result, where once again the exact coefficients can be determined explicitly. Calculating the exact coefficients requires taking the $k$-th derivative of a power series derived using a series expansion of the logarithm of the Gamma function, which naturally produces the complete Bell Polynomials. Whereas the expansion of the killed range is straightforward to establish using probabilistic tools in the case $z\in(0,1)$, the Tauberian Theorem requires it to hold on a Pac-Man-shaped domain $\Delta_{\phi,\eta}\subset\CC$, and this extension is dealt with essentially in Appendix \ref{appendix:complex analysis}.

The remainder of the paper is dedicated to the proof of Theorem~\ref{theorem:CLT}. We follow the dyadic strategy introduced by Le Gall \cite{LeGall1986} and further developed by Le Gall and Rosen \cite{LeGallRosen1991} for the range. Our main contribution here is showing that this strategy can be applied to the random walk sausage and identifying the exact dependence or lack thereof in $A$ of the fluctuations.

We first study the hitting time $T_{x+A}$ in Section~\ref{Asymptotic behaviour of T_A}. At leading order, it has the same asymptotic behavior as the hitting time of a single point: when $x$ is far away, the finite set $x+A$ is macroscopically indistinguishable from one point. However, the occupation time of $x+A$ contains an additional factor $|A|$, since every point of $A$ contributes to the occupation measure. In Section~\ref{Intersections of random walk sausages}, we introduce the intersection local time of two independent random walks $S, \widetilde S$
\begin{equation*}
    J_n^A=\sum_{x\in\mathbb Z^2}L_n^{x+A}\widetilde L_n^{x+A},
\end{equation*}
where for $B\subseteq\ZZ^2$, $L_n^B$ and $\widetilde{L}_n^B $ denote respectively the occupation measure of $S$ and $\widetilde{S}$ on the set $B$, namely $L_n^B:=\sum_{k=0}^n\1\left(S_k\in B \right)$.
We prove that
\begin{equation*}
    \frac{J_n^A}{n}
\Longrightarrow
|A|^2\alpha([0,1]^2),
\end{equation*}
where $\alpha$ is the mutual intersection local time of two independent planar Brownian motions. The factor $|A|^2$ comes from the two occupation measures. By denoting by $\I_n^{A}:=\left|(R_n+A)\cap(\widetilde{R}_n+A)\right|$, which is the intersection of two independent random walk sausages, the main comparison result is
\begin{equation*}
    \frac1n\left(
(\log n)^2\I_n^A
-
\frac{4\pi^2\sigma^4}{|A|^2}J_n^A
\right)
\longrightarrow0
\end{equation*}
in $L^2$. Therefore, the factor $|A|^2$ cancels and the rescaled intersection of the two sausages has the same limiting behavior as in the case $A=\{0\}$ of Le Gall and Rosen. Finally, in Section~\ref{Proof of the main result}, we insert this new intersection estimate into the dyadic decomposition of Le Gall and Rosen. This yields the fluctuation theorem for the random walk sausage and explains why its limiting fluctuations do not depend on $A$.

\color{black}

\subsection{Notation}

In this section, we introduce some notation that will be used in the remainder of the paper.

For any $x\in\ZZ^2$, we let $\PP_x$ denote the unique probability measure such that $S_0=x$ $\PP_x-a.s.$, and we write $\EE_x$ for the expectation under $\PP_x$. When $x=0$, we simply write $\PP$ and $\EE$.

Given a subset $B\subseteq\ZZ^2$, $T_B$ denotes the hitting time of by $S$, $i.e.$
$T_B:=\inf\left\{n\ge0\mid S_n\in B\right\}$. If $B=\{x\}$, we write $T_x$ instead of $T_{\{x\}}$. We also define the return times $T_B^+$ as $T_B^+:=\inf\left\{n\ge1\mid S_n\in B\right\}$

For $B\subseteq\ZZ^2$, $L_n^B$ denotes the occupation measure of $S$ on the set $B$, namely $L_n^B:=\sum_{k=0}^n\1\left(S_k\in B \right)$.

We define the Torus $\TT^2:=[-\pi,\pi]^2$, and its scaled version $\TT_n^2:=[-\pi\sqrt n,\pi\sqrt n]^2$. Also, we let $\mathbb D:=\{z\in\CC\mid |z|\le 1\}$ denote the unit disk in $\CC$.

Finally, we let $\mathbf 1$ denote the vector containing only ones, and $J:=\mathbf 1\mathbf 1^T$ denote the matrix containing only ones, which we view respectively as elements of $\CC^{A}$ and $\CC^{A\times A}$.

\section{The expansion of the mean volume} 
\label{development of the sausage}

In this section, we recall some important facts about potential theory in dimension $2$, following \cite{Lawler_Limic_2010}, Section $6.6$. In what follows, $A$ is a fixed finite subset of $\ZZ^2$.

Firstly, we define the potential kernel as
\begin{equation*}
    a(x):=\sum_{k\ge0}\left(\PP(S_k=0)-\PP(S_k=x)\right),\quad x\in\ZZ^2\backslash\{0\},
\end{equation*}
and $a(0)=0$. Then, we define the harmonic measure of $A$ as the limit 
\begin{equation*}
    \mu_A(x):=\lim_{|y|\rightarrow\infty}\PP_y(S_{T_A}=x),\quad x\in A.
\end{equation*}
The harmonic measure $\mu_A$ is a probability measure supported on $A$. For a measure $\mu$ on $A$, define its potential as 
\begin{equation*}
    U^\mu(x):=\sum_{y\in A}a(x-y)\mu(y).
\end{equation*}
For $\mu=\mu_A$, then the potential $U^{\mu_A}$ is constant on $A$ and thus we define the capacity of $A$, denoted $\mathrm{Cap}(A)$, by the quantity
\begin{equation*}
    \mathrm{Cap}(A):=U^{\mu_A}(x)=\sum_{y\in A}a(x-y)\mu_A(y),
\end{equation*}
where $x$ is an arbitrary element of $A$. Note that the Capacity of a singleton is always $0$. In the following proof, we suppose \ref{Assumption A1}, \ref{Assumption A2} and \ref{Assumption A3}. In particular, there is some $\delta>0$ such that $\EE[|S_1|^{2+\delta}]<\infty$.

\begin{proof}[Proof of Theorem \ref{theorem:expansion of the mean}]
    For $z\in(0,1)$, let $\xi_z$ be a geometric random variable of parameter $1-z$ that is independent of $S$ with values in $\mathbb N$. Then we have 
    \begin{align*}
        \EE\left[\R_{\xi_z}^{A}\right]&=\sum_{x\in\ZZ^2}\PP_x\left(\xi_z\ge T_A\right)\\
        &=\sum_{x\in\ZZ^2}\EE_x\left[\PP(\xi_z\ge T_A\mid T_A)\right]=\sum_{x\in\ZZ^2}\EE_x\left[z^{T_A}\right],
    \end{align*}
    since for any $k\ge0$, $\PP(\xi_z\ge k)=z^k$. For $z\in\mathring{\mathbb D}$, let
    \begin{equation*}
        F_A(z):=\sum_{n\ge0}z^n\EE[\R_n^{A}].
    \end{equation*}
    We introduce the weighted Green's function
    \begin{equation*}
        G_z(x,y):=\sum_{k\ge0}z^k\PP_x(S_k=y),
    \end{equation*}
    and the matrix $G_{z,A}:=(G_z(x,y))_{x,y\in A}$. We have the following identity.
    \begin{lemma}
    \label{lemma:idenetity for mean volume power series}
        For every $z\in\mathring{\mathbb D}$,
        \begin{equation}
        \label{eq: formula for F_A(z)}
            F_A(z)=\frac{\mathbf 1^TG_{z,A}^{-1}\mathbf 1}{(1-z)^2}.
        \end{equation}
    \end{lemma}
    \begin{proof}[Proof of Lemma \ref{lemma:idenetity for mean volume power series}]
        We start by noting that for $z\in(0,1)$, $\EE[\R_{\xi_z}^{A}]=(1-z)\sum_{n\ge0}z^n\EE[\R_n^{A}]=(1-z)F_A(z)$, in such a way that 
        \begin{equation}
    \label{eq:first expression of F_A}
        F_A(z)=\frac{1}{1-z}\sum_{x\in\ZZ^2}\EE_x\left[z^{T_A}\right].
    \end{equation}
    Since $F_A$ is analytic on $\mathring{\mathbb D}$, \eqref{eq:first expression of F_A} can be extended to $\mathring{\mathbb D}$. For $z\in\mathring{\mathbb D}$, $x,y\in A$, define 
    \begin{equation*}
        H_{z,A}(x,y):=\EE_x\left[z^{T_A^+}\1\left(S_{T_A^+}=y\right)\right].
    \end{equation*}
    By the first-return decomposition, we have for any $z\in\mathring{\mathbb D}$, $x,y\in A$,
    \begin{align*}
        G_z(x,y)&=\1(x=y)+\sum_{k\ge1}z^k\PP_x(S_k=y)\\
        &=\1(x=y)+\sum_{k\ge1}z^k\sum_{\ell=1}^k\sum_{a\in A}\PP_x\left(S_k=y,T_A^+=\ell, S_\ell=a\right)\\
        &=\1(x=y)+\sum_{a\in A}\sum_{k\ge1}z^k\sum_{\ell=1}^k\PP_x(T_A^+=\ell,S_\ell=a)\PP_a(S_{k-\ell}=y),
    \end{align*}
    by the Markov property. Changing the order of summation yields
    \begin{align*}
        G_z(x,y)&=\1(x=y)+\sum_{a\in A}\sum_{\ell\ge1}\sum_{k\ge0}z^{k+\ell}\PP_x(T_A^+=\ell,S_\ell=a)\PP_a(S_{k}=y)\\
        &=\1(x=y)+\sum_{a\in A}\left(\sum_{\ell\ge1}z^\ell\PP_x(T_A^+=\ell,S_\ell=a)\right)\bigg(\sum_{k\ge0}z^k\PP_a(S_k=y)\bigg)\\
        &=\1(x=y)+\sum_{a\in A}H_{z,A}(x,a)G_{z,A}(a,y).
    \end{align*}
    In matrix notation, this reads $G_{z,A}=\mathrm{Id}+H_{z,A}G_{z,A}$. In particular, $G_{z,A}$ is invertible and $G_{z,A}^{-1}=\mathrm{Id}-H_{z,A}$. By dominated convergence, $z\mapsto H_{z,A}$ is holomorphic on $\mathring{\mathbb D}$, and therefore so is $z\mapsto G_{z,A}^{-1}$. Let $e_{z,A}:=G_{z,A}^{-1}\mathbf 1$. For $x\in\ZZ^2$, we set $e_{z,A}(x)$ to be the element of the vector $e_{z,A}$ indexed by $x$ if $x\in A$ and $0$ otherwise. Therefore, viewing $e_{z,A}$ as a function on $\ZZ^2$, we write for any $x\notin A$ and $z\in\mathring{\mathbb D}$,
    \begin{equation*}
        u_z(x):=\sum_{y\in A}G_z(x,y)e_{z,A}(y)=\sum_{y\in A}e_{z,A}(y)\sum_{k\ge0}z^k\PP_x(S_k=y)=\EE_x\bigg[\sum_{k\ge0}z^k e_{z,A}(S_k)\bigg].
    \end{equation*}
    Since $e_{z,A}\equiv 0$ on $A^c$, we have $e_{z,A}(S_k)=0$ for $k<T_A$, and so by the Markov property at time $T_A$,
    \begin{equation*}
        u_z(x)=\EE_x\bigg[z^{T_A}\sum_{k\ge0}z^k e_{z,A}(S_{k+T_A})\bigg]=\EE_x\bigg[z^{T_A}\EE_{S_{T_A}}\bigg[\sum_{k\ge0}z^ke_{z,A}(S_k)\bigg]\bigg]=\EE_x\left[z^{T_A} u_z(S_{T_A})\right].
    \end{equation*}
    By our definition of $e_{z,A}$, we have $u_z\equiv 1$ on $A$, whence 
    \begin{equation*}
        \EE_x\left[z^{T_A}\right]=\sum_{y\in A}G_z(x,y)e_{z,A}(y).
    \end{equation*}
    Summing over $x\in\ZZ^2$, we have for any $z\in\mathring{\mathbb D}$,
    \begin{align}
        F_A(z)&=\frac{1}{1-z}\sum_{y\in A}e_{z,A}(y)\sum_{x\in\ZZ^2}G_z(x,y)=\frac{1}{1-z}\sum_{y\in A}e_{z,A}(y)\sum_{k\ge0}z^k\sum_{x\in\ZZ^2}\PP_x(S_k=0)\nonumber\\
        &=\frac{1}{(1-z)^2}\sum_{y\in A}e_{z,A}(y)=\frac{\mathbf 1^T G_{z,A}^{-1}\mathbf 1}{(1-z)^2},
    \end{align}
    which was our claim.
    \end{proof}

    We now introduce the complex domain $\Delta_{\phi,\eta}:=\{z\in\CC\mid |z|<1+\phi,\text{ }|\arg(z-1)|>\eta\}$ and let $L(z):=\mathrm{Log}(1/(1-z))$, where $\mathrm{Log}$ is the principle branch of the complex logarithm. The next step in the proof is to extend \eqref{eq: formula for F_A(z)} to $\Delta_{\phi,\eta}$ and obtain an asymptotic expansion of its right hand side as $z\rightarrow1$ in $\Delta_{\phi,\eta}$. For this reason, we state the following two Propositions, the proofs of which are deferred to Appendices \ref{appendix:proof of expansion of G_z(x)} and \ref{appendix:proof of G_A is invertible and with holomorphic inverse} respectively.
\begin{proposition}
\label{prop:expansion of G_z(x)}
    Let $\varepsilon\in(0,1\wedge\delta)$. Then for $\phi>0$ sufficiently small and $\eta\in(0,\pi/2)$, we have for every $z\in\Delta_{\phi,\eta}$, uniformly for $x\in A$,
    \begin{equation*}
        G_z(x)=\frac{L(z)}{2\pi\sigma^2}+\kappa_S-a(x)+O(|1-z|^{\varepsilon/2}),\quad z\rightarrow1,
    \end{equation*}
    where $\kappa_S$ is a constant that depends only on the distribution of $S$.
\end{proposition}
    \begin{proposition}
    \label{prop:G_A is invertible and with holomorphic inverse}
        Let $G_{z,A}:=(G_z(x,y))_{x,y\in A}$. Then for some $\phi>0$ sufficiently small, $G_{z,A}$ is invertible for every $z\in\Delta_{\phi,\eta}$. In particular, $z\mapsto G_{z,A}^{-1}$ is holomorphic on $\Delta_{\phi,\eta}$.
    \end{proposition}
These two results imply the following immediate corollary, the proof of which can be found in Appendix \ref{appendix:proof of dev asymp de 1^T G_(z,A) 1}.
\begin{corollary}        
\label{corollaire:dev asymp de 1^T G_(z,A) 1}
    Let $N\ge1$. Then for $\phi>0$ sufficiently small, as $z\rightarrow 1$ in $\Delta_{\phi,\eta}$,
    \begin{equation*}
        \mathbf 1^TG_{z,A}^{-1}\mathbf 1=\sum_{k=1}^N\frac{\left(\Cap(A)-\kappa_S\right)^{k-1}}{\lambda^k L(z)^k}+O\left(\frac{1}{L(z)^{N+1}}\right).
    \end{equation*}
\end{corollary}

    Thanks to Proposition \ref{prop:G_A is invertible and with holomorphic inverse}, so long as $\phi$ is sufficiently small, $z\mapsto\mathbf 1^T G_{z,A}^{-1}\mathbf 1$ is holomorphic on $\Delta_{\phi,\eta}$, and so we may extend \eqref{eq: formula for F_A(z)} analytically to $\Delta_{\phi,\eta}$. Letting $\lambda:=1/2\pi\sigma^2$, then according to Corollary \ref{corollaire:dev asymp de 1^T G_(z,A) 1}, we have 
\begin{align}
\label{eq:power series expression for E[R_n^A]}
    \sum_{n\ge0}\EE[\R_n^{A}]z^n&=(1-z)^{-2}L(z)^{-1}\left(\sum_{k=0}^{N-1}\frac{(\Cap(A)-\kappa_S)^k}{\lambda^{k+1}L(z)^k}+O\left(\frac{1}{L(z)^N}\right)\right).
\end{align}
Therefore, thanks to \cite{FlajoletOdlyzko}, Corollary $5$,
\begin{equation*}
    \EE\left[\R_n^{A}\right]=\sum_{k=1}^Nc_k\frac{n}{(\log n)^k}+O\left(\frac{n}{(\log n)^{N+1}}\right),
\end{equation*}
for some real constants $c_k$. To determine the value of the $c_k$, we use \cite{FlajoletOdlyzko} Theorem $3A$. Defining
\begin{equation*}
    e_k^{(\alpha,\gamma)}:=(-1)^k\binom{\gamma}{k}\Gamma(-\alpha)\frac{\d^k}{\d s^k}\left.\left(\frac{1}{\Gamma(-s)}\right)\right\vert_{s=\alpha},
\end{equation*}
as in \cite{FlajoletOdlyzko}, where $k\ge0$ and $\alpha,\gamma\in\CC\backslash\{0,1,2,\dots\}$, we have $c_1=1/\lambda$ and for any $k\in\{2,\dots N\}$,
\begin{align*}
    c_k&=\frac{1}{\lambda}\sum_{\ell=0}^{k-1}\left(\frac{\Cap(A)-\kappa_S}{\lambda}\right)^\ell e_{k-1-\ell}^{(-2,-1-\ell)}\\
    &=\frac{1}{\lambda}\sum_{\ell=0}^{k-1}\binom{k-1}{\ell}\left(\frac{\Cap(A)-\kappa_S}{\lambda}\right)^\ell\frac{\d^{k-1-\ell}}{\d s^{k-1-\ell}}\left.\left(\frac{1}{\Gamma(2-s)}\right)\right\vert_{s=0}\\
    &=\frac{1}{\lambda}\frac{\d^{k-1}}{\d s^{k-1}}\left.\left(\frac{\exp\left(\lambda^{-1}(\Cap(A)-\kappa_S)s\right)}{\Gamma(2-s)}\right)\right\vert_{s=0},
\end{align*}
by Leibniz' formula. Using the fact that $\log\Gamma (1-s)=C_Es+\sum_{\ell\ge2}\zeta(\ell)s^\ell/\ell$ for $|s|<1$, where $\zeta$ is the Riemann zeta function and that $\Gamma(2-s)=(1-s)\Gamma(1-s)$, we get
\begin{align*}
    \frac{\exp\left(\lambda^{-1}(\Cap(A)-\kappa_S)s\right)}{\Gamma(2-s)}&=\exp\bigg(\frac{\Cap(A)-\kappa_S}{\lambda}s-\log(1-s)-\log\Gamma(1-s)\bigg)\\
    &=\exp\bigg(\rho_{A,S}s+\sum_{\ell\ge2}\frac{1-\zeta(\ell)}{\ell}s^\ell\bigg)=\exp\bigg(\sum_{\ell\ge1}x_\ell\frac{s^\ell}{\ell!}\bigg),
\end{align*}
where we let $\rho_{A,S}:=1-C_E+(\Cap(A)-\kappa_S)/\lambda$ and we defined $x_1:=\rho_{A,S}$, $x_\ell:=(\ell-1)!(1-\zeta(\ell))$. By \eqref{eq:fundamental equation of bell polynomials}, we have
\begin{equation*}
    \exp\bigg(\sum_{\ell\ge1}x_\ell\frac{s^\ell}{\ell!}\bigg)=\sum_{\ell\ge0}B_\ell(x_1,\dots,x_\ell)\frac{s^\ell}{\ell!},
\end{equation*}
whence $c(1,A,S)=c_1=1/\lambda$ and 
\begin{equation*}
    c(k,A,S)=c_k=\frac{1}{\lambda}B_{k-1}\left(\rho_{A,S}, 1-\zeta(2), 2(1-\zeta(3)),\dots, (k-2)!(1-\zeta(k-1))\right),
\end{equation*}
    which is the announced result.
\end{proof}

\section{Asymptotic behaviour of normalized hitting times}\label{Asymptotic behaviour of T_A}

We wish to study the law of $T_{A+x}/|x|^2$ as $|x|\rightarrow\infty$. In order to do this, we work with the Laplace transform of $T_{A+x}/|x|^2$.

\begin{proposition}\label{lawofT}
    Let $A$ be a finite subset of $\ZZ^2$. Then for any $ x \in \mathbb{R}^d \setminus \{0\}$,
   \begin{equation}
   \label{limit:uniform convergence of scaled hitting times}
       \lim_{n \rightarrow \infty} \sup_{k\le n}\left|(\log n) \mathbb{P}(T_{\lfloor x\sqrt{n}\rfloor +A}\leq k)-2\pi \sigma^2\int_0^{k/n} \d s \ p_s(x)\right|=0,
   \end{equation}
   where $p_t(x)=(2\pi\sigma^2 t)^{-1}\exp\left(-|x|^2/2 \sigma^2t\right)$. Furthermore, for any $\varepsilon>0$, $x\ne 0$, we have 
   \begin{equation}
       \label{ineq:bound for dominated convergence}
       \sup_{n\ge1}\ (\log n)\PP(T_{\lfloor x\sqrt n\rfloor+A}\le n)\lesssim |x|^\varepsilon+|x|^{-\varepsilon}.
   \end{equation}
\end{proposition}
\begin{proof}
    For any $q>0$, we introduce the measure
    \begin{equation*}
        G_q(A):=\sum_{n\ge0}\exp\left(-q n\right)\PP(S_n\in A).
    \end{equation*}
The strong Markov property at time $T_A$ gives
\begin{equation}
\label{eq:Markov pour G_lambda(A)}
    G_q(A)=\EE\left[\exp(-qT_A)G_q(A-S_{T_A})
    \right].
\end{equation}

Fix $\lambda>0$ and $z\in\ZZ^2\setminus\{0\}$. Applying \eqref{eq:Markov pour G_lambda(A)} to the set $A+z$ and decomposing according to the entrance point, we obtain
\begin{equation}
\label{eq:strong Markov entrance decomposition}
\begin{aligned}
    G_{\frac{\lambda}{|z|^2}}(A+z)&=\EE\left[
        \exp\left(-\frac{\lambda}{|z|^2}T_{A+z}\right)
        G_{\frac{\lambda}{|z|^2}}\left(A+z-S_{T_{A+z}}\right)\right]\\
    &=
    \sum_{y\in A}\EE\left[\exp\left(-\frac{\lambda}{|z|^2}T_{A+z}\right)\1({S_{T_{A+z}}=y+z)}\right]
    G_{\frac{\lambda}{|z|^2}}\left(A-y\right).
\end{aligned}
\end{equation}
For $y\in A$, let
\begin{equation*}
    w_z(y):=\EE\left[\exp\left(-\frac{\lambda}{|z|^2}T_{A+z}\right)\1(S_{T_{A+z}}=y+z)\right].
\end{equation*}
Notice that
\begin{equation}
\label{eq:sum entrance weights}
    \sum_{y\in A}w_z(y)=
    \EE\left[\exp\left(-\frac{\lambda}{|z|^2}T_{A+z}\right)\right].
\end{equation}

For every fixed $u\in\ZZ^2$, the local central limit theorem yields
\begin{equation*}
    G_q(u)
    \sim
    G_q(0)
    \sim
    \frac{\log(1/q)}{2\pi\sigma^2},
    \qquad q\rightarrow0.
\end{equation*}
Since $A$ is finite and
\begin{equation*}
    \frac{\log(|z|^2/\lambda)}{\log |z|^2}
    \longrightarrow 1,
\end{equation*}
we deduce, uniformly in $y\in A$, that
\begin{equation}
\label{eq:asymptotics Green entrance set}
    G_{\frac{\lambda}{|z|^2}}(A-y)=\sum_{a\in A}G_{\frac{\lambda}{|z|^2}}(a-y)=
    \frac{|A|}{2\pi\sigma^2}\log|z|^2\bigl(1+o(1)\bigr).
\end{equation}
Define
\begin{equation*}
    \varepsilon_z(y)
    :=
    1-
    \frac{2\pi\sigma^2}{|A|\log |z|^2}G_{\frac{\lambda}{|z|^2}}(A-y)
\end{equation*}
and
\begin{equation*}
    \delta_z:=\max_{y\in A}|\varepsilon_z(y)|.
\end{equation*}
By \eqref{eq:asymptotics Green entrance set}, we have
$\delta_z\to0$ as $|z|\to\infty$. On the other hand, reasoning as in
\cite[Proposition~5.2]{LeGallRosen1991}, for every fixed
$a\in\ZZ^2$, $\lambda>0$ and $\xi\in\mathbb S^1$,
\begin{equation*}
    G_{\lambda/|z|^2}(z+a)
    \underset{\substack{|z|\to\infty\\z/|z|\to\xi}}
    {\longrightarrow}
    \int_0^\infty \exp(-\lambda t)p_t(\xi)\,\d t.
\end{equation*}
Consequently,
\begin{equation}
\label{eq:scaled Green set limit}
    G_{\frac{\lambda}{|z|^2}}(A+z)
    =
    \sum_{a\in A}G_{\frac{\lambda}{|z|^2}}(z+a)
    \underset{\substack{|z|\to\infty\\z/|z|\to\xi}}
    {\longrightarrow}
    |A|\int_0^\infty \exp(-\lambda t)p_t(\xi)\,\d t.
\end{equation}
We now add and subtract the leading term in
\eqref{eq:asymptotics Green entrance set}. Using
\eqref{eq:strong Markov entrance decomposition} and
\eqref{eq:sum entrance weights}, we obtain
\begin{align}
    \log |z|^2\EE\left[\exp(-\frac{\lambda}{|z|^2}T_{A+z})\right]
    &=
    \sum_{y\in A}w_z(y)\log |z|^2
    \nonumber\\
    &=
    \frac{2\pi\sigma^2}{|A|}
    \sum_{y\in A}w_z(y)G_{\frac{\lambda}{|z|^2}}(A-y)
    +
    \log |z|^2\sum_{y\in A}w_z(y)\varepsilon_z(y)
    \nonumber\\
    &=
    \frac{2\pi\sigma^2}{|A|}G_{\frac{\lambda}{|z|^2}}(A+z)+ \log |z|^2\sum_{y\in A}w_z(y)\varepsilon_z(y).
\label{eq:Laplace transform decomposition}
\end{align}
The second term satisfies
\begin{equation}
\label{eq:bound Laplace remainder}
\begin{aligned}
    \left| \log |z|^2\sum_{y\in A}w_z(y)\varepsilon_z(y)\right|
    &\leq
    \delta_z\log |z|^2\sum_{y\in A}w_z(y)\\
    &=
    \delta_z\log |z|^2
    \EE\left[\exp(-\frac{\lambda}{|z|^2}T_{A+z})\right].
\end{aligned}
\end{equation}
Since $\delta_z\to0$, for all sufficiently large $|z|$, and using \eqref{eq:scaled Green set limit},
\begin{equation*}
    \log |z|^2\EE\left[\exp\left(-\lambda T_{A+z}/|z|^2\right)\right]\leq\frac{2\pi\sigma^2}{|A|}G_{\frac{\lambda}{|z|^2}}(A+z)(1-\delta_z)^{-1}=O(1),
\end{equation*}
whence
\begin{equation*}
    \left|\log |z|^2\sum_{y\in A}w_z(y)\varepsilon_z(y)\right|\leq\delta_z\log |z|^2\EE\left[\exp(-\frac{\lambda}{|z|^2}T_{A+z})\right]\longrightarrow0.
\end{equation*}
Therefore, using again \eqref{eq:scaled Green set limit}, we conclude
that, for every $\lambda>0$ and $\xi\in\mathbb S^1$,
\begin{equation}
\label{eq:limit Laplace normalized hitting time}
    (\log |z|^2)
    \EE\left[
        \exp\left(
            -\lambda\frac{T_{A+z}}{|z|^2}
        \right)
    \right]
    \underset{\substack{|z|\to\infty\\z/|z|\to\xi}}
    {\longrightarrow}
    2\pi\sigma^2
    \int_0^\infty
        \exp(-\lambda t)p_t(\xi)\,\d t.
\end{equation}
In particular, by letting $\mu_z^A(\d t):=(\log |z|^2)\PP(T_{A+z}/|z|^2\in\d t)$ and $\mu_\xi(\d t):=2\pi\sigma^2p_t(\xi)\d t$, then we have shown that for any $\lambda>0$, as $|z|\rightarrow\infty$, $z/|z|\rightarrow\xi$, 
    \begin{equation*}
        \int_0^\infty\exp(-\lambda t)\mu_z^{A}(\d t)\longrightarrow\int_0^\infty\exp(-\lambda t)\mu_\xi(\d t).
    \end{equation*}
    Therefore, $\mu_z^{A}$ converges in the vague topology to $\mu_\xi$ as $|z|\rightarrow\infty$, $z/|z|\rightarrow\xi$, and so for any $t>0$,
    \begin{equation*}
        (\log|z|^2)\PP\left(T_{A+z}\le t|z|^2\right)\longrightarrow2\pi\sigma^2\int_0^t\d s\ p_s(\xi).
    \end{equation*}
    Since both the left and right hand sides of the previous equation are increasing in $t$, then by Dini's second theorem, for any $T>0$,
    \begin{equation}
    \label{limit:uniform convergence 1}
        \sup_{t\in[0,T]}\left|(\log|z|^2)\PP\left(T_{A+z}\le t|z|^2\right)-2\pi\sigma^2\int_0^t\d s\ p_s(\xi)\right|\underset{\substack{|z|\rightarrow\infty \\ z/|z|\rightarrow\xi }}\longrightarrow0.
    \end{equation}
     Letting $x\in\RR^2$, $x\ne 0$, $n\ge1$ and $z_n=\lfloor x\sqrt n\rfloor$, we have for any $\alpha\in[0,1]$,
    \begin{align*}
        (\log n)\PP\left(T_{A+\lfloor x\sqrt n\rfloor}\le \alpha n\right)=&\ \frac{\log n}{\log |z_n|^2}\left(\log|z_n|^2\PP\left(T_{A+z_n}\le \alpha\frac{n}{|z_n|^2}|z_n|^2\right)\right)\\
        \underset{n\rightarrow\infty}\longrightarrow&\ 2\pi\sigma^2\int_0^{\alpha/|x|^2}\d s \ p_s(x/|x|),
    \end{align*}
    where we used the fact that $n/|z_n|^2\rightarrow1/|x|^2$, that $\log$ is slowly varying at infinity and \eqref{limit:uniform convergence 1}. Finally, by definition of $p_s$, note that 
    \begin{equation*}
        p_s(x/|x|)=|x|^2p_{s|x|^2}(x),
    \end{equation*}
    and so
    \begin{equation*}
        \underset{n\rightarrow\infty}\lim(\log n)\PP\left(T_{A+\lfloor x\sqrt n\rfloor}\le \alpha n\right)=2\pi\sigma^2\int_0^\alpha\d s\ p_s(x).
    \end{equation*}
    Once again, we obtain the uniform convergence result \eqref{limit:uniform convergence of scaled hitting times} by Dini's second theorem. It remains to prove \eqref{ineq:bound for dominated convergence}. Recalling \cite{LeGallRosen1991}, Lemma 6.2 (iii), we have 
    \begin{equation*}
        (\log n)\PP(T_{\lfloor x\sqrt n\rfloor}\le n)\lesssim|x|^\varepsilon+|x|^{-\varepsilon}.
    \end{equation*}
    Therefore,
    \begin{equation*}
        (\log n)\PP(T_{\lfloor x\sqrt n\rfloor +A\le n})=\left(\frac{\PP(T_{\lfloor x\sqrt n\rfloor+A}\le n}{\PP(T_{\lfloor x\sqrt n\rfloor}\le n)}\right)(\log n)\PP(T_{\lfloor x\sqrt n\rfloor}\le n)\lesssim |x|^\varepsilon+|x|^{-\varepsilon},
    \end{equation*}
    since \eqref{limit:uniform convergence of scaled hitting times} implies that the sequence $\PP(T_{\lfloor x\sqrt n\rfloor+A}\le n)/\PP(T_{\lfloor x\sqrt n\rfloor}\le n)$ is bounded.
\end{proof}

\begin{proposition}\label{prop limit:uniform convergence of local time}
        Let $A$ be a finite subset of $\ZZ^2$. Then for any $ x \in \mathbb{R}^d \setminus \{0\}$,
   \begin{equation}
   \label{limit:uniform convergence of local time}
       \lim_{n \rightarrow \infty} \sup_{k\le n}\left|\EE\left[L_k^{A+\lfloor x\sqrt n\rfloor}\right]-|A|\int_0^{k/n} \d s \ p_s(x)\right|=0. 
   \end{equation}
   Furthermore, for any $\varepsilon>0$, $x\ne 0$, we have 
   \begin{equation}
       \label{ineq:bound for dominated convergence of local time}
       \sup_{n\ge 1}\EE\left[L_n^{A+\lfloor x\sqrt n\rfloor}\right]\lesssim |x|^\varepsilon+|x|^{-\varepsilon}.
   \end{equation}
    \end{proposition}
    \begin{proof}
    Let $\alpha\in[0,1]$. For every $a\in A$, applying the Markov propery at time $T_{a+\lfloor x\sqrt n\rfloor}$ we get
    \begin{align}
    \label{eq:markov property for local time}
        \EE\left[L_{\alpha n}^{A+\lfloor x\sqrt n\rfloor}\right]&=\sum_{a\in A}\EE\left[L_{\alpha n}^{a+\lfloor x\sqrt n\rfloor}\right]\nonumber\\
        &=\sum_{a\in A}\sum_{\ell=0}^{\alpha n}\PP\left(T_{a+\lfloor x\sqrt n\rfloor}=\ell\right)\EE\left[L_{\alpha n-\ell}^0\right].
    \end{align}
    Let $\varepsilon_n$ be a non-increasing sequence such that $\varepsilon_n=o(1)$ and $\log\varepsilon_n=o(\log n)$. Since
    \begin{equation*}
        \EE[L_n^0]=(1+o(1))\frac{\log n}{2\pi\sigma^2},
    \end{equation*}
    then $\EE[L_n^0]$ is slowly varying at infinity and in particular, since $n\varepsilon_n\rightarrow\infty$, we have 
    \begin{equation*}
        \EE[L_{\alpha\varepsilon_nn}^0]=(1+o(1))\frac{\log(\varepsilon_nn)}{2\pi\sigma^2}=(1+o(1))\frac{\log n}{2\pi\sigma^2}\left(1+\frac{\log\varepsilon_n}{\log n}\right)=(1+o(1))\frac{\log n}{2\pi\sigma^2}.
    \end{equation*}
    Therefore, uniformly for $\ell\in[\![0,\alpha(1-\varepsilon_n)n]\!]$,
    \begin{equation*}
        \EE[L_{\alpha n-\ell}^0]=(1+o(1))\frac{\log n}{2\pi\sigma^2}.
    \end{equation*}
    Fixing $a\in A$, we write
    \begin{align*}
        \sum_{\ell=0}^{\alpha n}\PP\left(T_{a+\lfloor x\sqrt n\rfloor}=\ell\right)\EE\left[L_{\alpha n-\ell}^0\right]&=\left(\sum_{\ell=0}^{\alpha(1-\varepsilon_n)n}+\sum_{\ell=\alpha(1-\varepsilon_n)n+1}^{\alpha n}\right)\PP\left(T_{a+\lfloor x\sqrt n\rfloor}=\ell\right)\EE\left[L_{\alpha n-\ell}^0\right]\\
        &=:\Sigma_n^1+\Sigma_n^2.
    \end{align*}
    By what we just saw,
    \begin{align*}
        \Sigma_n^1&=(1+o(1))\frac{\log n}{2\pi\sigma^2}\PP\left(T_{a+\lfloor x\sqrt n\rfloor}\le\alpha(1-\varepsilon_n)n\right)\\
        &=(1+o(1))\int_0^{\alpha(1-\varepsilon_n)}\d s\ p_s(x),
    \end{align*}
    where we used \eqref{limit:uniform convergence of scaled hitting times} in the second equality. Furthermore,
    \begin{align*}
        \Sigma_n^2&\le\EE[L_{\alpha\varepsilon_nn}^0]\sum_{\ell=\alpha(1-\varepsilon_n)n+1}^{\alpha n}\PP\left(T_{a+\lfloor x\sqrt n\rfloor}=\ell\right)\\
        &=(1+o(1))\frac{\log n}{2\pi\sigma^2}\PP\left(\alpha(1-\varepsilon_n)n<T_{a+\lfloor x\sqrt n\rfloor}\le \alpha n\right)\\
        &=(1+o(1))\int_{\alpha(1-\varepsilon_n)}^\alpha\d s\ p_s(x)=O(\varepsilon_n),
    \end{align*}
    where again we used \eqref{limit:uniform convergence of scaled hitting times} in the second equality. Combining the previous calculations, we observe that
    \begin{equation*}
        \EE\left[L_{\alpha n}^{A+\lfloor x\sqrt n\rfloor}\right]=|A|(1+o(1))\int_{0}^{\alpha(1-\varepsilon_n)}\d s\ p_s(x)+O(\varepsilon_n),
    \end{equation*}
    whence for any $\alpha\in[0,1]$,
    \begin{equation*}
        \lim_{n\rightarrow\infty}\EE\left[L_{\alpha n}^{A+\lfloor x\sqrt n\rfloor}\right]=|A|\int_0^\alpha\d s\ p_s(x).
    \end{equation*}
    As in the proof of Proposition \ref{lawofT}, the uniform convergence result \eqref{limit:uniform convergence of local time} is a consequence of Dini's second theorem, and \eqref{ineq:bound for dominated convergence of local time} is deduced similarly.
    \end{proof}

\section{Intersections of random walk sausages}\label{Intersections of random walk sausages}

Throughout this section, we will let $S,\widetilde S$ be two independent random walks in $\ZZ^2$ satisfying assumptions \ref{Assumption A1} and \ref{Assumption A2}. We also introduce the notation 
\begin{equation*}
    S_t^{(n)}:=\frac{S_{\lfloor nt\rfloor}}{\sqrt n},\quad \widetilde S_t^{(n)}:=\frac{\widetilde S_{\lfloor nt\rfloor}}{\sqrt n},\quad t\in[0,1],
\end{equation*}
and let $B$, $\widetilde B$ be two independent standard brownian motions in $\RR^2$ such that we have the following joint convergence in the product $J_1$ topology
\begin{equation}
    \label{limit:Donsker's theorem}
    \left(\left(S_t^{(n)}\right)_{t\in[0,1]},\left(\widetilde S_t^{(n)}\right)_{t\in[0,1]}\right)\underset{n\rightarrow\infty}\implies\left(\left(B_{\sigma^2t}\right)_{t\in[0,1]},\left(\widetilde B_{\sigma^2t}\right)_{t\in[0,1]}\right).
\end{equation}
To simplify notation, we will write $S^{(n)}$ instead of $\left(S_t^{(n)}\right)_{t\in[0,1]}$ and $W$ instead of $\left(B_{\sigma^2t}\right)_{t\in[0,1]}$, and similarly for $\widetilde S^{(n)}$ and $\widetilde W$. We fix $A\subseteq\ZZ^2$ and define the cross section
\begin{equation*}
    \I_n^{A}:=\left|(\R_n+A)\cap(\widetilde\R_n+A)\right|.
\end{equation*}
The quantity $\I_n^A$ is the number of distinct sites that the two random walk sausages $\R_n+A$ and $\widetilde\R_n+A$ have in common. We start by estimating the moments of $\I_n^A$.

\begin{proposition}
\label{prop:bound for moments of intersections}
For any $k\ge1$, 
    \begin{equation*}
        \EE\left[\left(\I_n^A\right)^k\right]\le C_{A,k}\frac{n^k}{(\log n)^{2k}}
    \end{equation*}
\end{proposition}

\begin{proof}
    By the definition of $\I_n^{A}$, there is some $C_{A,k}>0$ such that 
    \begin{align*}
        \EE\left[\left(\I_n^{A}\right)^k\right]&\le\sum_{\substack{a_1,\dots,a_k\in A \\ b_1,\dots,b_k\in A}}\EE\left[\prod_{j=1}^k\left|\R_n^{a_j}\cap\widetilde\R_n^{b_j}\right|\right]\\
        &\le\sum_{\substack{a_1,\dots,a_k\in A \\ b_1,\dots,b_k\in A}}\prod_{j=1}^k\EE\left[\left|\R_n^{a_j}\cap\widetilde\R_n^{b_j}\right|^k\right]^{1/k},
    \end{align*}
    where we applied Hölder's inequality. We fix $a,b\in A$ and introduce the random time $\tau_n:=\inf\{j\le n\mid S_j\in\widetilde\R_n^{b-a}\}$, with the convention $\inf\emptyset=\infty$. On the event $\{\tau_n=\infty\}$, we have $\left|\R_n^{a}\cap\widetilde\R_n^b\right|=0$, and on the event $\{\tau_n<\infty\}$, considering $\tilde\tau_n$, the smallest time $j$ for which $\widetilde S_j$ started from $b$ is equal to $S_{\tau_n}$ (the fact that $\tau_n<\infty$ automatically implies that $\tilde\tau_n<\infty$), we bound
    \begin{align*}
        \EE\left[\left|\R_n^{a}\cap\widetilde \R_n^{b}\right|^k\right]&\le C_k\left( \EE\left[\left|\left(S(\tau_n,n)+a\right)\cap(\widetilde S(0,\tilde\tau_n)+b)\right|^k\right]\right.\\
        &\hspace{6em}\left.+\EE\left[\left|\left(S(\tau_n,n)+a\right)\cap(\widetilde S(\tilde\tau_n,n)+b)\right|^k\right]\right)\\
        &\le 2C_k\EE\left[\left|\R_n\cap\widetilde\R_n\right|^k\right],
    \end{align*}
    where we used the fact that the law of $\widetilde S$ is stable by time reversal and that $(i,j)\mapsto |S(0,i)\cap \widetilde S(0,j)|$ is increasing. Recalling \cite{LeGall1986}, Theorem $5.1$, we have $\EE[|\R_n\cap\widetilde\R_n|^k]\le C n^k/(\log n)^{2k}$ for some $C>0$, whence
    \begin{align*}
        \EE\left[\left(\I_n^{A}\right)^k\right]&\le2\sum_{\substack{a_1,\dots,a_k\in A \\ b_1,\dots,b_k\in A}}\EE\left[\left|\R_n\cap\widetilde\R_n\right|^k\right]\le2C|A|^{2k}\frac{n^k}{(\log n)^{2k}},
    \end{align*}
    whence the result.
\end{proof}

We introduce 
\begin{equation*}
    \J_n^{A}:=\sum_{x\in\ZZ^2}L_n^{x+A}\widetilde L_n^{x+A}.
\end{equation*}
The first result we establish is that the discrete intersection local time $\J_n^{A}$ conveniently scaled converges in distribution to the intersection local time of the limiting brownian motions.

\begin{proposition}
\label{prop:convergence of discrete ILT to continuous ILT}
We have the following convergence in distribution
    \begin{equation*}
        \frac{1}{n}\J_n^{A}\underset{n\rightarrow\infty}{\implies}|A|^2\alpha\left([0,1]^2\right).
    \end{equation*}
\end{proposition}

\begin{proof}
    We write
    \begin{equation*}
        \J_n^{A}=\sum_{a,b\in A}\sum_{i,j=0}^n\sum_{x\in\ZZ^2}\1\left(S_i=x+a\right)\1\left(\widetilde S_j=x+b\right)=\sum_{a,b\in A}\sum_{i,j=0}^n\1\left(S_i-a=\widetilde S_j-b\right).
    \end{equation*}
    Using the fact that $\1(x=y)=(2\pi)^{-2}\int_{\TT^2}\d\xi\exp\left(-\i\langle\xi,x-y\rangle\right)$, we therefore get
    \begin{align*}
        \frac{1}{n}\J_n^{A}&=\frac{1}{n}\sum_{a,b\in A}\sum_{i,j=0}^n\frac{1}{(2\pi)^2}\int_{\TT^2}\d\xi\exp\left(-\i\left\langle\xi,S_i-\widetilde S_j+b-a\right\rangle\right)\\
        &=\sum_{a,b\in A}\int_{[0,1]^2}\d u\d v\frac{1}{(2\pi)^2}\int_{\TT_n^2}\d\xi\exp\left(-\i\left\langle\xi,S_u^{(n)}-\widetilde S_v^{(n)}+n^{-1/2}(b-a)\right\rangle\right)
    \end{align*}
    Let $f_\varepsilon$ denote the heat kernel in $\RR^2$. Its Fourier transform is given by 
    \begin{equation}
    \label{eq:Fourier transform of the heat kernel}
        \hat{f}_\varepsilon(\xi):=\exp\left(-\varepsilon|\xi|^2\right),\quad \xi\in\RR^2,
    \end{equation}
    and by Fourier inversion we have 
    \begin{equation}
    \label{eq:Fourier inversion for the heat kernel}
        f_\varepsilon(x)=\frac{1}{(2\pi)^2}\int_{\RR^2}\d\xi\exp(-\i\langle\xi,x\rangle)\hat{f}_\varepsilon(\xi).
    \end{equation}
    We now introduce a smoothed version of $\J_n^{A}$ denoted $\J_{n,\varepsilon}^{A}$ for all $\varepsilon>0$ and defined by 
    \begin{equation*}
       \frac{1}{n}\J_{n,\varepsilon}^{A}:= \sum_{a,b\in A}\int_{[0,1]^2}\d u\d v\frac{1}{(2\pi)^2}\int_{\TT_n^2}\d\xi\exp\left(-\i\left\langle\xi,S_u^{(n)}-\widetilde S_v^{(n)}+n^{-1/2}(b-a)\right\rangle\right)\hat{f}_\varepsilon(\xi).
    \end{equation*}
    For all $x\in\RR^2$, 
    \begin{equation*}
        \left|\int_{(\TT_n^2)^c}\d\xi\exp(-\i\langle\xi,x\rangle)\hat{f}_\varepsilon(\xi)\right|\le\int_{\pi\sqrt n}^\infty\d t\exp\left(-\varepsilon t^2\right)\le C\exp\left(-c\varepsilon n\right),
    \end{equation*}
    where $c,C$ are positive constants not depending on $x,n$. Combining this with (\ref{eq:Fourier inversion for the heat kernel}) and writing $\int_{\TT_n^2}=\int_{\RR^2}-\int_{(\TT_n^2)^c}$ yields
    \begin{equation*}
        \frac{1}{n}\J_{n,\varepsilon}^{A}=\sum_{a,b\in A}\int_{[0,1]^2}\d u\d v f_\varepsilon\left(S_u^{(n)}-\widetilde S_v^{(n)}+n^{-1/2}(b-a)\right) + O(\exp(-c\varepsilon n)).
    \end{equation*}
    Since $f_\varepsilon$ is smooth, \eqref{limit:Donsker's theorem} therefore implies that
    \begin{align*}
        \left(S^{(n)},\widetilde S^{(n)}, n^{-1}\J_{n,\varepsilon}^{A}\right)&\underset{n\rightarrow\infty}\implies\left(W,\widetilde W,\sum_{a,b\in A}\int_{[0,1]^2}\d u\d v\ f_\varepsilon\left(W_{u}-\widetilde W_{v}\right)\right)\\
        &=\left(W,\widetilde W,|A|^2\alpha_{\varepsilon}\left([0,1]^2\right)\right).
    \end{align*}
    It is well known that for $p\ge1$,
    \begin{equation*}
        \alpha_\varepsilon([0,1]^2)\underset{\varepsilon \rightarrow0}{\overset{L^p}{\longrightarrow}}\alpha\left([0,1]^2\right).
    \end{equation*}
    Finally, by \cite{Rosen1990}, Lemma 1, there exist $c,\delta>0$ such that for sufficiently large $n$,
    \begin{equation*}
        n^{-2}\EE\left[\left(\J_n^{A}-\J_{n,\varepsilon}^{A}\right)^2\right]\le C_A n^{-2}\EE\left[\left(\J_n^{\{0\}}-\J_{n,\varepsilon}^{\{0\}}\right)^2\right]\le C_A\varepsilon^\delta.
    \end{equation*}
    Combining this bound with the two previous convergence results concludes the proof.
\end{proof}

\begin{proposition}
\label{prop:intersections are close to ILT}We have the following convergence in $L^2$
    \begin{equation*}
        \frac{1}{n}\left((\log n)^2\I_n^{A}-\frac{4\pi^2\sigma^4}{|A|^2}\J_n^{A}\right)\underset{n\rightarrow\infty}{\overset{L^2}\longrightarrow}0.
    \end{equation*}
\end{proposition}

\begin{proof}
    Let $K>0$. We introduce 
    \begin{equation*}
        \I_{n,K}^{A}:=\left|(\R_n+A)\cap(\widetilde\R_n+A)\cap(K^2\TT_n^2)\right|
    \end{equation*}
    and
    \begin{equation*}
        \J_{n,K}^{A}:=\sum_{x\in \ZZ^2\cap K^2\TT_n^2}L_n^{x+A}\widetilde L_n^{x+A}
    \end{equation*}
    As in \cite{LeGallRosen1991} Proposition $6.3$, it is sufficient to show that
    \begin{equation}\label{cutoff for intersections are close to ILT}
        \lim_{n \rightarrow \infty }\frac{1}{n^2}\EE\bigg[\bigg((\log n)^2\I_{n,K}^{A}-\bigg(\frac{2\pi\sigma^2}{|A|}\bigg)^2\J_{n,K}^{A}\bigg)^2\bigg]=0,
    \end{equation}
    as the cutoff can be removed thanks to Donsker's Theorem. To prove \eqref{cutoff for intersections are close to ILT}, we write
    \begin{align*}
        &\frac{1}{n^2}\EE\bigg[\bigg((\log n)^2\I_{n,K}^{A}-\bigg(\frac{2\pi\sigma^2}{|A|}\bigg)^2\J_{n,K}^{A}\bigg)^2\bigg]\\
        &\hspace{3em}=\frac{(\log n)^4}{n^2}\EE\left[\left(\I_{n,K}^{A}\right)^2\right]-2\left(\frac{2\pi\sigma^2\log n}{|A| n}\right)^2\EE\left[\I_{n,K}^{A}\J_{n,K}^{A}\right]+\frac{1}{n^2}\left(\frac{2\pi\sigma^2}{|A|}\right)^4\EE\left[\left(\J_{n,K}^{A}\right)^2\right].
    \end{align*}
    We study each term in separate steps and show that they have the same asymptotic behaviour.

    \textit{Step 1:} By writing 
    \begin{equation*}
        \I_{n,K}^{A}=\sum_{x\in\ZZ^2\cap K^2\TT_n^2}\1\left(x\in\R_n+A\right)\1\left(x\in\widetilde\R_n+A\right)
    \end{equation*}
     and introducing the notation 
     \begin{equation*}
         \mathcal B_K^{(n)}:=\left\{x\in\RR^2,|\lfloor x\sqrt n\rfloor|\le K\sqrt n\right\},\quad \mathcal B_K:=\left\{x\in\RR^2, |x|\le K\right\},
     \end{equation*}
     we have 
    \begin{align*}
        \EE\left[\left(\I_{n,K}^A\right)^2\right]&=\sum_{x,y\in\ZZ^2\cap K^2\TT_n^2}\PP\left(T_{x+A},T_{y+A}\le n\right)\PP\left(\widetilde T_{x+A},\widetilde T_{y+A}\le n\right)\\
        &=\int_{|x|,|y|\le K\sqrt n}\d x\d y\ \PP\left(T_{\lfloor x\rfloor+A},T_{\lfloor y\rfloor +A}\le n\right)\PP\left(\widetilde T_{\lfloor x\rfloor+A},\widetilde T_{\lfloor y\rfloor+A}\le n\right)\\
        &=n^2\int_{\mathcal B_K^{(n)}\times\mathcal B_K^{(n)}}\d x\d y\ \PP\left(T_{\lfloor x\sqrt n\rfloor+A},T_{\lfloor y\sqrt n\rfloor+A}\le n\right)\PP\left(\widetilde T_{\lfloor x\sqrt n\rfloor+A},\widetilde T_{\lfloor y\sqrt n\rfloor+A}\le n\right).
    \end{align*}
    We let
    \begin{equation*}
        \Phi_n(x,y):=\log(n)^2\PP\left(T_{\lfloor x\sqrt n\rfloor+A},T_{\lfloor y\sqrt n\rfloor+A}\le n\right),
    \end{equation*}
    and define $\widetilde\Phi_n(x,y)$ analogously in such a way that 
    \begin{equation*}
        \frac{(\log n)^4}{n^2}\EE\left[\left(\I_{n,K}^A\right)^2\right]=\int_{\mathcal B_K^{(n)}\times\mathcal B_K^{(n)}}\d x\d y\ \Phi_n(x,y)\widetilde\Phi_n(x,y).
    \end{equation*}
    Let $A,B\subseteq\ZZ^2$. Then
    \begin{equation*}
        \PP\left(T_A,T_B\le n\right)=\PP\left(T_A\le T_B\le n\right)+\PP\left(T_B< T_A\le n\right).
    \end{equation*}
We first see that
\begin{equation*}
    \begin{split}
         \mathbb{P}\left(T_{\lfloor x \sqrt{n}\rfloor+A}\leq T_{\lfloor y \sqrt{n}\rfloor+A}\leq n \right)
        &=\sum_{k=0}^n \mathbb{P}\left( T_{\lfloor x \sqrt{n}\rfloor+A}=k, k\leq T_{\lfloor y \sqrt{n}\rfloor+A}\leq n \right).\\
    \end{split}
\end{equation*}
We use the fact that for any $ k \leq n$,
\begin{equation*}\label{hitting}
    \1\left(k< T_{\lfloor y \sqrt{n} \rfloor +A} \leq n \right) =\1\left( T_{\lfloor y \sqrt{n} \rfloor +A}\circ \theta_k \leq n-k\right) \left( 1-\1\left(T_{\lfloor y \sqrt{n}\rfloor +A} \leq k\right)\right),
\end{equation*}
and let
\begin{equation*}
    \Phi_n^1(x,y):=\left(\log n\right)^2\sum_{k=0}^n \PP\left( T_{\lfloor x \sqrt{n}\rfloor+A}=k,T_{\lfloor y \sqrt{n} \rfloor +A}\circ \theta_k \leq n-k\right),
\end{equation*}
\begin{equation*}
    E_n^1(x,y):= \left(\log n\right)^2\sum_{k=0}^n \PP\left( T_{\lfloor x \sqrt{n}\rfloor+A}=k,T_{\lfloor y \sqrt{n}\rfloor+A}\leq k,T_{\lfloor y \sqrt{n} \rfloor +A}\circ \theta_k \leq n-k\right),
\end{equation*}
so that 
\begin{equation*}
    \begin{split}
        \Phi_n(x,y) &=\Phi_n^1(x,y)+\Phi_n^1(y,x)-E_n^1(x,y)-E_n^1(y,x).
    \end{split}
\end{equation*}
Using the Markov property,
\begin{align*}
    \Phi_n^1(x,y)&= \left(\log n\right)^2 \sum_{k=0}^n \mathbb{E}\left[\mathds{1}(T_{\lfloor x \sqrt{n}\rfloor+A}=k)\mathbb{P}_{S_k}\left(T_{\lfloor y\sqrt{n }\rfloor+A}\leq n-k \right)\right]\\
    &=\left(\log n\right)^2\sum_{k=0}^n\sum_{a \in A} \mathbb{P}\left( S_{k}=\lfloor x \sqrt{n}\rfloor-a\mid T_{\lfloor x\sqrt{n}\rfloor+A}=k\right)\\
    & \hspace{3em}\times\mathbb{P}\left(T_{\lfloor x \sqrt{n}\rfloor+A}=k\right)\mathbb{P}_{\lfloor x \sqrt{n}\rfloor-a}\left(T_{\lfloor y\sqrt{n }\rfloor+A}\leq n-k \right).
\end{align*}
As a consequence of Proposition \ref{lawofT}, we have 
\begin{equation*}
    \PP_{\lfloor x\sqrt n\rfloor-a}\left(T_{\lfloor y\sqrt n\rfloor+A}\le n-k\right)=(1+o(1))\PP_{\lfloor x\sqrt n\rfloor}\left(T_{\lfloor y\sqrt n\rfloor+A}\le n-k\right),
\end{equation*}
uniformly for $a\in A$, $k\le n$. This yields
\begin{align*}
    \Phi_n^1(x,y)&=(\log n)^2\sum_{k=0}^n\sum_{a\in A}\mathbb{P}\left( S_{k}=\lfloor x \sqrt{n}\rfloor-a\mid T_{\lfloor x\sqrt{n}\rfloor+A}=k\right)\\
    & \hspace{3em}\times\mathbb{P}\left(T_{\lfloor x \sqrt{n}\rfloor+A}=k\right)(1+o(1))\mathbb{P}_{\lfloor x \sqrt{n}\rfloor}\left(T_{\lfloor y\sqrt{n }\rfloor+A}\leq n-k \right)\\
    &=(1+o(1))(\log n)^2\sum_{k=0}^n\mathbb{P}\left(T_{\lfloor x \sqrt{n}\rfloor+A}=k\right)\mathbb{P}_{\lfloor x \sqrt{n}\rfloor}\left(T_{\lfloor y\sqrt{n }\rfloor+A}\leq n-k \right).
\end{align*}
Therefore, once again thanks to Proposition \ref{lawofT}, we get
\begin{equation*}
    \lim_{n\rightarrow\infty}\Phi_n^1(x,y)=4\pi^2\sigma^4\int_0^1\d s\ p_s(x)\int_{0}^{1-s}\d t \ p_t (y-x).
\end{equation*}
Furthermore, 
\begin{align*}
    E_n^1(x,y)=\left(\log n\right)^2 \sum_{k=0}^n \EE\left[\1\left(T_{\lfloor x \sqrt{n}\rfloor +A}=k\right)\1\left( T_{\lfloor y \sqrt{n}\rfloor +A}\leq k\right)\PP_{S_k}\left(T_{\lfloor y \sqrt{n}\rfloor +A}\leq n-k\right)\right],
\end{align*}
and the second statement of Proposition \ref{lawofT} guarantees for every $x,y\in\RR^2$ the existence of $C_{x,y}$ such that 
\begin{equation*}
    \sum_{k=0}^n \EE\left[\1\left(T_{\lfloor x \sqrt{n}\rfloor +A}=k\right)\1\left( T_{\lfloor y \sqrt{n}\rfloor +A}\leq k\right)\PP_{S_k}\left(T_{\lfloor y \sqrt{n}\rfloor +A}\leq n-k\right)\right] \leq \frac{C_{x,y}}{\left(\log n\right)^3}.
\end{equation*}
In particular, for any $x,y\in\RR^2$,
\begin{equation*}
    \lim_{n \rightarrow \infty } E_n^1(x,y)=0
\end{equation*}
Combining these two estimates, we get  
\begin{equation*}
    \lim_{n \rightarrow \infty} \Phi_n(x,y)=4\pi^2 \sigma^4\left(\int_0^1 \d s \ p_s(x)\int_0^{1-s}\d t \  p_t(y-x) + \int_0^1 \d s \ p_s(y)\int_0^{1-s}\d t \  p_t(x-y)\right)
\end{equation*}
Moreover, \eqref{ineq:bound for dominated convergence} implies that there exists $C_{\varepsilon, A}>0$ such that 
\begin{equation}\label{bound}
    \Phi_n(x,y)\leq C_{\varepsilon, A}\left(|y-x|^{\varepsilon}+|y-x|^{-\varepsilon}\right)\left[\left(|x|^{\varepsilon}+|x|^{-\varepsilon}\right)+ \left(|y|^{\varepsilon}+|y|^{-\varepsilon}\right)\right]
\end{equation}
The same argument applies to $\widetilde{\Phi}_n(x,y)$. We can now apply the dominated convergence theorem, since the right-hand-side of \eqref{bound} is square integrable over $\mathcal B_K\times\mathcal B_K$ as soon as $\varepsilon$ is sufficiently small, and obtain
\begin{align}
        \lim_{n \rightarrow \infty}\frac{(\log n)^4}{n^2}\EE\left[\left(\I_{n,K}^{A}\right)^2\right]&= 16\pi^4 \sigma^8 \int_{\mathcal B_K\times\mathcal B_K}\d x\d y\ \left(\int_0^1 \d s \ p_s(x)\int_0^{1-s}\d t \  p_t(y-x)\right.\nonumber\\
        &\hspace{13em}\left.+ \int_0^1 \d s \ p_s(y)\int_0^{1-s}\d t \  p_t(x-y)\right)^2\nonumber
\end{align}
\textit{Step 2:} We now study the term $\EE\left[\I_n^{A}\J_n^{A}\right]$. Performing the same scaling as in \textit{Step 1}, we get
\begin{equation*}
   \left(\frac{2\pi\sigma^2\log n}{|A| n}\right)^2\EE\left[\I_{n,K}^{A}\J_{n,K}^{A}\right]=\int_{\mathcal B_K^{(n)}\times\mathcal B_K^{(n)}}\d x\d y\ \Psi_n(x,y)\tilde\Psi_n(x,y),
\end{equation*}
where we let
\begin{equation*}
    \Psi_n(x,y):=\frac{2 \pi \sigma^2\log n}{|A|}\EE\bigg[\1\left(\lfloor x\sqrt n\rfloor\in\R_n+A\right)\sum_{k=0}^n\1\left(S_k\in\lfloor y\sqrt n\rfloor + A\right)\bigg],\quad x,y\in\RR^2,
\end{equation*}
and define $\tilde\Psi_n(x,y)$ analogously by replacing $\R_n$ and $S_k$ by $\tilde\R_n$ and $\tilde S_k$ respectively. We now divide the sum in $\Psi_n(x,y)$ into two parts, and write $\Psi_n(x,y)=:\Psi_n^1(x,y)+\Psi_n^2(x,y)$, where $\Psi_n^1(x,y)$ and $\Psi_n^2(x,y)$ are defined as follows.
\begin{align*}
    \Psi_n^1(x,y):=\frac{2 \pi \sigma^2\log n}{|A|}\EE\bigg[\1\left(\lfloor x\sqrt n\rfloor\in\R_n+A\right)\sum_{k=T_{\lfloor x \sqrt{n}\rfloor +A}}^n\1\left(S_k\in\lfloor y\sqrt n\rfloor + A\right)\bigg],
\end{align*}
\begin{align*}
    \Psi_n^2(x,y):=\frac{2 \pi \sigma^2\log n}{|A|} \sum_{k=0}^n \PP\left( S_k \in \lfloor y \sqrt{n}\rfloor +A, k< T_{\lfloor x \sqrt{ n}\rfloor +A}\leq n\right).
\end{align*}
Looking at $\Psi_n^1(x,y)$ first, we see that
\begin{align*}
    \left(\frac{2 \pi \sigma^2\log n}{|A|}\right)^{-1}\Psi_n^1(x,y)&=\sum_{\ell=0}^n\EE\bigg[\1\left(T_{\lfloor x\sqrt n\rfloor+A}=\ell\right)\sum_{k=\ell}^n\1\left(S_k\in\lfloor y\sqrt n\rfloor+A\right)\bigg]\\
    &=\sum_{\ell=0}^n\EE\bigg[\1\left(T_{\lfloor x\sqrt n\rfloor+A}=\ell\right)\EE_{S_\ell}\bigg[\sum_{k=0}^{n-\ell}\1\left(S_k\in\lfloor y\sqrt n\rfloor+A\right)\bigg]\bigg]\\
    &=\sum_{a\in A}\sum_{\ell=0}^n\PP\left( S_{\ell}=\lfloor x \sqrt{n}\rfloor -a \mid T_{\lfloor x \sqrt{ n} \rfloor +A=\ell}\right)\\
    &\hspace{5em}\times\PP\left(T_{\lfloor x\sqrt n\rfloor+A}=\ell\right)\EE_{\lfloor x\sqrt n\rfloor-a}\left[L_{n-\ell}^{\lfloor y \sqrt{n}\rfloor+A}\right],
\end{align*}
Using Proposition \ref{prop limit:uniform convergence of local time}, we have that, uniformly for $a \in A$ and $k \leq n$,
\begin{equation}\label{locat time inv}
    \EE_{\lfloor x\sqrt n\rfloor-a}\left[L_{n-k}^{\lfloor y \sqrt{n}\rfloor+A}\right]=(1+o(1)) \EE_{\lfloor x\sqrt n\rfloor}\left[L_{n-k}^{\lfloor y \sqrt{n}\rfloor+A}\right].
\end{equation}
This implies that
\begin{align*}
    \left(\frac{2 \pi \sigma^2\log n}{|A|}\right)^{-1}\Psi_n^1(x,y)&= (1+o(1)) \sum_{\ell=0}^n \PP\left(T_{\lfloor x\sqrt n\rfloor+A}=\ell\right)\EE_{\lfloor x\sqrt n\rfloor}\left[L_{n-\ell}^{\lfloor y \sqrt{n}\rfloor+A}\right].
\end{align*}
Thanks to Proposition \ref{lawofT} and Proposition \ref{prop limit:uniform convergence of local time}, we get
\begin{equation}
    \lim_{n \rightarrow \infty} \Psi_n^1(x,y)=4\pi^2\sigma^4\int_0^1\d s\ p_s(x)\int_{0}^{1-s}\d t \ p_t (y-x).
\end{equation}

For $\Psi_n^2(x,y)$, we again use \eqref{hitting}, which allows us to write $\Psi_n^2(x,y)$ as
\begin{align*}
     \Psi_n^2(x,y)&:= \Psi_n^{3}(x,y)- E_n^2(x,y),
\end{align*}
where
\begin{equation*}
     \Psi_n^{3}(x,y):=\left(\frac{2 \pi \sigma^2\log n}{|A|}\right)
     \sum_{k=0}^n \PP\left(S_k \in \lfloor y \sqrt{n}\rfloor +A, T_{x \sqrt{n}+A}\circ \theta_k \leq n-k\right),
\end{equation*}
and 
\begin{equation*}
    E_n^2(x,y):=  \left(\frac{2 \pi \sigma^2\log n}{|A|}\right)\sum_{k=0}^n \PP\left( S_k \in \lfloor y \sqrt{n}\rfloor +A, T_{\lfloor x \sqrt{ n}\rfloor +A}\leq k, T_{\lfloor x \sqrt{ n}\rfloor +A}\circ \theta_k \leq n-k\right).
\end{equation*}
For $\Psi_n^{3}(x,y)$,
\begin{align*}
     \Psi_n^{3}(x,y)& =\left(\frac{2 \pi \sigma^2\log n}{|A|}\right)\sum_{k=0}^n \EE\left[\1\left( S_k \in \lfloor y \sqrt{n}\rfloor+A\right)\PP_{S_k}\left(T_{\lfloor x \sqrt{n}\rfloor+A}\leq n-k\right)\right]\\
     & = \left(\frac{2 \pi \sigma^2\log n}{|A|}\right)\sum_{a \in A}\sum_{k=0}^n \PP\left(S_k=\lfloor y \sqrt{n }\rfloor+a\mid T_{\lfloor y \sqrt{n }\rfloor+A}=k\right)\times \PP\left(T_{\lfloor y \sqrt{n }\rfloor+A}=k\right)\\
     & \hspace{3em} \times\PP_{\lfloor y \sqrt{n }\rfloor+a}\left( T_{\lfloor x \sqrt{n }\rfloor+A}\leq n-k\right)
\end{align*}     
An argument analogous to that used for $\Phi_n^1(x,y)$ yields
\begin{equation}
    \lim_{n \rightarrow \infty } \Psi_n^3(x,y)=4\pi^2\sigma^4\int_0^1\d s\ p_s(y)\int_{0}^{1-s}\d t \ p_t (x-y).
\end{equation}
And
\begin{align*}
    E_n^2(x,y)&= \left(\frac{2 \pi \sigma^2\log n}{|A|}\right)\sum_{k=0}^n \PP\left( S_k \in \lfloor y \sqrt{n}\rfloor +A, T_{\lfloor x \sqrt{ n}\rfloor +A}\leq k, T_{\lfloor x \sqrt{ n}\rfloor +A}\circ \theta_k \leq n-k\right)\\
    & = \left(\frac{2 \pi \sigma^2\log n}{|A|}\right)\sum_{k=0}^n \EE\left[ \1\left(T_{\lfloor x \sqrt{n}\rfloor +A}\leq k\right) \1\left( S_k \in \lfloor y \sqrt{n}\rfloor +A\right)\PP_{S_k}\left( T_{\lfloor x \sqrt{ n}\rfloor +A}\leq n-k\right)\right]
\end{align*}
Proceeding as in Step 1 and applying Proposition \ref{lawofT}, we obtain
\begin{equation*}
    \lim_{n \rightarrow \infty}E_n^2(x,y)=0
\end{equation*}
Combining these estimates, treating $ \widetilde{\Psi}_n(x,y)$ analogously, and using \eqref{ineq:bound for dominated convergence of local time} to apply the dominated convergence Theorem, we conclude that
\begin{align}
    \lim_{n \rightarrow \infty} \left(\frac{2\pi\sigma^2\log n}{|A| n}\right)^2\EE\left[\I_{n,K}^{A}\J_{n,K}^{A}\right]&=16\pi^4 \sigma^8 \int_{\mathcal B_K\times\mathcal B_K}\d x\d y\ \left(\int_0^1 \d s \ p_s(x)\int_0^{1-s}\d t \  p_t(y-x)\right.\nonumber\\
        &\hspace{10em}\left.+ \int_0^1 \d s \ p_s(y)\int_0^{1-s}\d t \  p_t(x-y)\right)^2\nonumber
\end{align}
\textit{Step 3:} We now turn to the remaining term. Following the same steps, we get
\begin{align*}
    \frac{1}{n^2}\left(\frac{2\pi\sigma^2}{|A|}\right)^4\EE\left[\left(\J_{n,K}^{A}\right)^2\right]& =\int_{\mathcal B_K^{(n)}\times\mathcal B_K^{(n)}}\d x\d y\ \Xi_n(x,y)\widetilde\Xi_n(x,y). 
\end{align*}
where 
\begin{align*}
    \Xi_n(x,y)&:= \left(\frac{2\pi \sigma^2}{|A|}\right)^2\sum_{k=0}^n \sum_{\ell=0}^n \EE\left[ \1\left( S_k \in \lfloor x \sqrt{n}\rfloor +A\right)\1\left( S_{\ell} \in \lfloor y \sqrt{n}\rfloor +A\right)\right]
\end{align*}
The quantity $\widetilde{\Xi}_n(x,y)$ is defined analogously. We first treat $\Xi_n(x,y)$, since the analysis of $\widetilde{\Xi}_n(x,y)$ follows along the same lines. We decompose $\Xi_n(x,y)$ as follows:
\begin{align*}
    \Xi_n(x,y)&=\Lambda_n^1(x,y)+\Lambda_n^2(x,y)
\end{align*}
where 
\begin{align*}
    \Lambda_n^1(x,y):=\left(\frac{2\pi\sigma^2}{|A|}\right)^2  \sum_{k=0}^n\EE\left[ \1\left( S_k \in \lfloor x \sqrt{n}\rfloor +A\right)\sum_{\ell=k}^n \1\left( S_{\ell}\in \lfloor y \sqrt{n} \rfloor +A\right)\right]
\end{align*}
\begin{align*}
    \Lambda_n^2(x,y):= \left(\frac{2\pi\sigma^2}{|A|}\right)^2  \sum_{k=0}^n\EE\left[ \1\left( S_k \in \lfloor x \sqrt{n}\rfloor +A\right)\sum_{\ell=0}^{k-1} \1\left( S_{\ell} \in \lfloor y \sqrt{n} \rfloor +A\right)\right]
\end{align*}
The first term can be handled by applying the Markov property at $S_k$, while the second term can be handled by applying the Markov property at $S_{\ell}$.
\begin{equation*}
    \begin{split}
    \Lambda_n^1(x,y)&= \left(\frac{2\pi\sigma^2}{|A|}\right)^2 \sum_{k=0}^n \left[\1\left( S_k \in \lfloor x \sqrt{n}\rfloor +A\right)\EE_{S_k}\left[\sum_{\ell=0}^{n-k} \1\left( S_{\ell} \in \lfloor y \sqrt{n}\rfloor +A\right)\right] \right]\\
    & =\left(\frac{2\pi\sigma^2}{|A|}\right)^2 \sum_{a\in A} \sum_{k=0}^n \PP\left(S_k =\lfloor x \sqrt{ n}\rfloor +a\mid S_k \in \lfloor x \sqrt{ n}\rfloor +A\right) \times \PP\left(S_k \in \lfloor x \sqrt{n}\rfloor+A\right)\\
    & \hspace{3em} \times \EE_{\lfloor x \sqrt{ n}\rfloor +a}\left[ L_{n-k}^{\lfloor y \sqrt{n}\rfloor +A}\right]\\
    \end{split}
 \end{equation*}
 Using again \eqref{locat time inv}, we get
 \begin{align*}
     \Lambda_n^1(x,y)&=\left(\frac{2\pi\sigma^2}{|A|}\right)^2 \sum_{k =0}^n \PP\left(S_k \in \lfloor x \sqrt{n}\rfloor+A\right)\EE_{\lfloor x \sqrt{ n}\rfloor}\left[ L_{n-k}^{\lfloor y \sqrt{n}\rfloor +A}\right]+o(1)
 \end{align*}
It follows immediately from Proposition \ref{prop limit:uniform convergence of local time} that 
\begin{align*}
    \lim_{n \rightarrow \infty} \Lambda_n^1(x,y)& =4\pi^2\sigma^4\int_0^1\d s\ p_s(x)\int_{0}^{1-s}\d t \ p_t (y-x).
\end{align*}
And $\Lambda_n^2(x,y)$ can be treated analogously.
We conclude that
\begin{align}
    \lim_{n \rightarrow \infty}\frac{1}{n^2}\left(\frac{2\pi\sigma^2}{|A|}\right)^4\EE\left[\left(\J_{n,K}^{A}\right)^2\right]&=16\pi^4 \sigma^8 \int_{\mathcal B_K\times\mathcal B_K}\d x\d y\ \left(\int_0^1 \d s \ p_s(x)\int_0^{1-s}\d t \  p_t(y-x)\right.\nonumber\\
        &\hspace{10em}\left.+ \int_0^1 \d s \ p_s(y)\int_0^{1-s}\d t \  p_t(x-y)\right)^2\nonumber
\end{align}
by following the same arguments as in Step 1. This completes the proof of \eqref{cutoff for intersections are close to ILT}. 
\end{proof}
As a corollary of Propositions \ref{prop:convergence of discrete ILT to continuous ILT} and \ref{prop:intersections are close to ILT}, we obtain the following 

\begin{proposition}
\label{prop:joint conergence of walks and intersections}
Let $X$, $\widetilde X$ be two independent centered random walks in $\ZZ^2$ with covariance $\sigma^2\mathrm{Id}$ such that the following convergence holds in the usual product $J_1$ topology
\begin{equation*}
    n^{-1/2}\left(X_{\lfloor nt\rfloor},\widetilde X_{\lfloor nt\rfloor}\right)_{t\in[0,1]}\underset{n\rightarrow\infty}\implies\left(W_{\sigma^2t},\widetilde W_{\sigma^2t}\right)_{t\in[0,1]},
\end{equation*}
where $W$, $\widetilde W$ are independent standard brownian motions in $\RR^2$. Then by letting $\R_n$, $\widetilde\R_n$ denote the respective ranges of $X$, $\widetilde X$ and $\I_n^{A}=|(\R_n+A)\cap(\widetilde\R_n+A)|$, we have
    \begin{equation*}
        \left((X_t^{(n)})_{t\le 1},(\widetilde X_t^{(n)})_{t\le 1},\frac{(\log n)^2}{n}\I_n^{A},\right)\underset{n\rightarrow\infty}\implies\left( (W_t)_{t\le\sigma^2}, (\widetilde W_t)_{t\le\sigma^2}, 4\pi^2\sigma^2\alpha\left([0,1]^2\right)\right),
    \end{equation*}
    where $X_t^{(n)}:= n^{-1/2}X_{\lfloor nt\rfloor}$ and where $\alpha$ is the intersection local time of $W$ and $\widetilde W$.
\end{proposition}

\section{The fluctuations of the sausage volume}\label{Proof of the main result}

In this section, we prove Theorem \ref{theorem:CLT}. Once the proof is complete, we obtain as a by-product a strong law of large numbers for $\R_n^{A}$

\subsection{The fluctuations}

We introduce the shorthand notation $\R_n^{A}:=|\R_n+A|$. Since $(A\cup B)+C=(A+C)\cup(B+C)$ for any subsets $A,B,C\subseteq\ZZ^2$, then by writing $\R_n=S(0,n/2)\cup S(n/2,n)$, we get
\begin{align*}
    \R_n^{A}&=\left|\left(S\left(0,\frac n2\right)+A\right)\cup\left(S\left(\frac n2,n\right)+A\right)\right|\\
    &=\left|\left(S\left(0,\frac n2\right)+A\right)\right|+\left|\left(S\left(\frac n2,n\right)+A\right)\right|\\
    &\hspace{5em}-\left|\left(S\left(0,\frac n2\right)+A\right)\cap\left(S\left(\frac n2,n\right)+A\right)\right|.
\end{align*}
Iterating this procedure $p$ times yields the following identity
\begin{align}
    \R_n^{A}&=\sum_{j=1}^{2^p}\left|S\left(\frac{j-1}{2^p}n,\frac{j}{2^p}n\right)+A\right|\nonumber\\
    &\hspace{3em}-\sum_{q=1}^p\sum_{\ell=1}^{2^{q-1}}\left|\left(S\left(\frac{2\ell-2}{2^q}n,\frac{2\ell-1}{2^q}n\right)+A\right)\cap\left(S\left(\frac{2\ell-1}{2^q}n,\frac{2\ell}{2^q}n\right)+A\right)\right|\nonumber\\
    &=:\sum_{j=1}^{2^p}\R_n^{A}(j,p)-\sum_{q=1}^p\sum_{\ell=1}^{2^{q-1}}\I_n^{A}(\ell,q).\label{eq:dyadic decomposition of R_n^A}
\end{align}
Note that for every $p\ge1$, the $(\R_n^{A}(j,p))_{1\le j\le 2^p}$ are $i.i.d.$ and distributed as $\R_{n/2^p}^{A}$. For $q\ge1$, $\ell\le2^{q-1}$, we introduce
\begin{equation*}
    D_{\ell,q}:=\left[\frac{2\ell-2}{2^q},\frac{2\ell-1}{2^q}\right]\times\left[\frac{2\ell-1}{2^q},\frac{2\ell}{2^q}\right],
\end{equation*}
and for any $p\ge1$,
\begin{equation*}
    \Gamma_p:=\sum_{q=1}^p\sum_{\ell=1}^{2^{q-1}}\left(\alpha(D_{\ell,q})-\EE[\alpha(D_{\ell,q})]\right).
\end{equation*}
Then it is well known (\cite{LeGall1986}) that 
\begin{equation*}
    \Gamma_p\underset{p\rightarrow\infty}{\overset{L^2}\longrightarrow}\gamma_1.
\end{equation*}
By proposition \ref{prop:joint conergence of walks and intersections}, we have
\begin{equation*}
    \frac{(\log n)^2}{n}\sum_{q=1}^p\sum_{\ell=1}^{2^{q-1}}\overline\I_n^{A}\underset{n\rightarrow\infty}\implies4\pi^2\sigma^2\Gamma_p.
\end{equation*}
Furthermore, as in \cite{LeGall1986}, the formula \eqref{eq:dyadic decomposition of R_n^A} and the bound of Proposition \ref{prop:bound for moments of intersections} yield the variance bound
\begin{equation}
\label{ineq:variance bound}
    \mathrm{Var}(\R_n^{A})\lesssim\frac{n^2}{(\log n)^4}.
\end{equation}
In particular,
\begin{align*}
    \mathrm{Var}\left(\frac{(\log n)^2}{n}\sum_{j=1}^{2^p}\R_n^{A}(j,p)\right)&=\frac{(\log n)^4}{n^2}\sum_{j=1}^{2^p}\mathrm{Var}(\R_{n/2^p}^{A})\lesssim 2^p\frac{(\log n)^4}{n^2}\frac{(n/2^p)^2}{\left(\log(n/2^p)\right)^4}\lesssim 2^{-p}.
\end{align*}
Combining all the results, then after centering and multiplying by $(\log n)^2/n$ in \eqref{eq:dyadic decomposition of R_n^A}, letting $n\rightarrow\infty$ and then $p\rightarrow\infty$, we obtain
\begin{equation*}
    \frac{(\log n)^2}{n}(\R_n^{A}-\EE[\R_n^{A}])\underset{n\rightarrow\infty}\implies-4\pi^2\sigma^2\gamma_1.
\end{equation*}

\subsection{The law of large numbers}

As an easy consequence of preceding results, we obtain the following strong law of large numbers for $\R_n^{A}$.

\begin{theorem}
    \begin{equation*}
        \frac{\log n}{n}\R_n^{A}\underset{n\rightarrow\infty}{\overset{a.s.}{\longrightarrow}}2\pi\sigma^2
    \end{equation*}
\end{theorem}

\begin{proof}
    We follow the proof of \cite{LeGallRosen1991}, Theorem $6.9$. Let $r>1$ and $m_n=\lfloor r^n\rfloor$. Then for any $\varepsilon>0$, using Theorem \ref{theorem:expansion of the mean} and the variance bound \eqref{ineq:variance bound},
    \begin{align*}
        \sum_{n\ge1}\PP\left(\left|\R_{m_n}^{A}-\EE\left[\R_{m_n}^{A}\right]\right|>\varepsilon\EE\left[\R_{m_n}^{A}\right]\right)\le\sum_{n\ge1}\frac{\mathrm{Var}\left(\R_{m_n}^{A}\right)}{\varepsilon^2\EE\left[\R_{m_n}^{A}\right]^2}\lesssim \sum_{n\ge1}\frac{1}{n^2}<\infty,
    \end{align*}
    and so by the Borel-Cantelli lemma,
    \begin{equation*}
        \frac{\log \lfloor r^n\rfloor}{\lfloor r^n\rfloor}\R_{\lfloor r^n\rfloor}^{A}\underset{n\rightarrow\infty}{\overset{a.s.}{\longrightarrow}}2\pi\sigma^2.
    \end{equation*}
    Since $\R_n^{A}$ is increasing in $n$, this directly implies that $a.s.$, 
    \begin{equation*}
        2\pi\sigma^2/r\le\liminf \frac{\log n}{n}\R_n^{A}\le\limsup\frac{\log n}{n}\le 2\pi\sigma^2 r,
    \end{equation*}
    and this is sufficient to conclude since $r$ can be taken arbitrarily close to $1$.
\end{proof}

\appendix

\section{Bell polynomials and some explicit coefficients}
\label{appendix:bell polynomials}

Bell polynomials were introduced by Bell in \cite{Bell1934} to study exponentials of power series. For every $n\ge0$, we let $B_n$ denote the $n$-th complete exponential Bell polynomial (see \cite{Bell1934}, section II, with the notation $\phi_n$), where for each $n\ge1$, $B_n$ is a polynomial of the variables $x_1,\dots,x_n\in\CC$ and $B_0=1$. Then, we define the $B_n$ recursively by the recurrence
\begin{equation*}
    B_{n+1}(x_1,\dots, x_{n+1})=\sum_{\ell=0}^n \binom{n}{\ell}x_{\ell+1}B_{n-\ell}(x_1,\dots, x_{n-\ell}).
\end{equation*}
The purpose of this definition is then that the $B_n$ satisfy the fundamental relation
\begin{equation}
\label{eq:fundamental equation of bell polynomials}
    \exp\left(\sum_{\ell\ge1}x_\ell\frac{s^\ell}{\ell!}\right)=\sum_{\ell\ge0} B_\ell(x_1,\dots,x_\ell)\frac{s^\ell}{\ell!},
\end{equation}
and have the following explicit epression
\begin{equation*}
    B_n(x_1,\dots,x_n)=n!\sum_{\substack{k_1,\dots, k_n \\ k_1+2k_2+\dots+nk_n=n}}\frac{x_1^{k_1}\cdots x_n^{k_n}}{(1!)^{k_1}\cdots (n!)^{k_n} k_1!\cdots k_n!}.
\end{equation*}
Note in particular that the number of terms in $B_n(x_1,\dots,x_n)$ is the number of partitions of $n$. The three first complete exponential Bell polynomials are
\begin{align*}
    B_1(x_1)&=x_1,\\
    B_2(x_1,x_2)&=x_1^2+x_2,\\
    B_3(x_1,x_2,x_3)&=x_1^3+3x_1x_2+x_3,
\end{align*}
which leads to the following expressions for the first four coefficients in the asymptotic expansion of $\EE[\R_n^{A}]$,
\begin{align*}
    c(1,A,S)&=2\pi\sigma^2,\\
    c(2,A,S)&=2\pi\sigma^2\rho_{A,S},\\
    c(3,A,S)&=2\pi\sigma^2(\rho_{A,S}^2+1-\pi^2/6),\\
    c(4,A,S)&=2\pi\sigma^2(\rho_{A,S}^3+3\rho_{A,S}(1-\pi^2/6)+2(1-\zeta(3))).
\end{align*}
For more information about Bell polynomials and the link to partitions of $n$, one can consult \cite{Comtet1974} and the tables therein.

\section{Random walk and potential theory estimates}
\subsection{Killed green's function in $\Delta_{\phi,\eta}$}
\label{appendix:complex analysis}
The aim of this section is to establish three Propositions that are used in the proof of Theorem~\ref{theorem:expansion of the mean}. Recall the definition of the domain $\Delta_{\phi,\eta}:=\{z\in\CC\mid |z|< 1+\phi,\text{ }|\mathrm{arg}(z-1)|>\eta\}$. Let 
\begin{equation*}
    G_z(x)=\sum_{n\ge0}z^n\PP(S_n=x),\quad a(x)=\sum_{n\ge0}\left(\PP(S_n=0)-\PP(S_n=x)\right),
\end{equation*}
and define the characteristic function $\varphi$ of $S$ as 
\begin{equation*}
    \varphi(\xi):=\EE[\exp(-\i\langle\xi,S_1\rangle].
\end{equation*}
Also, recall that $L(z)=\mathrm{Log}(1/(1-z))$ for $z\in\Delta_{\phi,\eta}$, where $\mathrm{Log}$ is the principal branch of the complex logarithm. 

\subsubsection{Proof of Proposition \ref{prop:expansion of G_z(x)}}
\label{appendix:proof of expansion of G_z(x)}
    Our expansion is similar to the one of \cite{BassRosen}, Lemma $A.1$, but we are interested in showing it on the domain $\Delta_{\phi,\eta}$ instead of on the real line. Throughout the proof, in order to lighten notation, we shall write $\lambda:=1/2\pi\sigma^2$. For $z\in\CC$, $|z|<1$, Fourier inversion yields
    \begin{equation}
    \label{eq:fourier inversion for G_z(x)}
        G_z(x)=\frac{1}{(2\pi)^2}\int_{\TT^2}\d\xi\frac{\exp(\i\langle x,\xi\rangle)}{1-z\varphi(\xi)}.
    \end{equation}
    Let $r\in(0,\pi)$. If $|\xi|\ge r$, then by strong aperiodicity, we have $\sup_{|\xi|\ge r}|\varphi(\xi)|<1$, whence for $z\in\Delta_{\phi,\eta}$,
    \begin{equation}
    \label{ineq:bound of 1-z phi(xi) for large xi}
        |1-z\varphi(\xi)|\ge1-(1+\phi)\sup_{|\xi|\ge r}|\varphi(\xi)|>0,
    \end{equation}
    so long as $\phi$ is chosen small enough. To treat the case $|\xi|<r$, start by noting that our hypothesis on $S$ yield
    \begin{equation}
    \label{eq:asymptotic expansion of characteristic function}
        \varphi(\xi)=1-\frac{\sigma^2}{2}|\xi|^2+R(\xi),
    \end{equation}
    where $\varepsilon\in(0,1\wedge\delta)$ and $R(\xi)=O(|\xi|^{2+\varepsilon})$. Let $s:=1-z$ and write $s=\rho \exp(\i\theta)$. Then for any $t\ge0$, we have 
    \begin{equation*}
        |s+t|^2=\rho^2+t^2+2\rho t\cos(\theta).
    \end{equation*}
    Since $z\in\Delta_{\phi,\eta}$, we have $|\theta|\le\pi-\eta$, and so $\cos(\alpha)\ge-\cos(\eta)$, whence 
    \begin{equation*}
        |s+t|^2\ge t^2+\rho^2-2t\rho\cos(\eta)\ge (\sin(\eta/2))^2(|s|+t)^2.
    \end{equation*}
    In particular, there is some constant $c_\eta$ such that 
    \begin{equation}
    \label{ineq:lower bound for |s+q|}
        \left|s+\frac{\sigma^2}{2}|\xi|^2\right|\ge c_\eta\left(|s|+|\xi|^2\right).
    \end{equation}
    Now, we write
    \begin{equation*}
        1-z\varphi(\xi)=s+\frac{\sigma^2}{2}|\xi|^2-s\frac{\sigma^2}{2}|\xi|^2+(1-s)R(\xi).
    \end{equation*}
    Take $z\in\Delta_{\phi,\eta}$ sufficiently close to $1$ so that $|s|\le\phi$. For this choice, we have for some $C>0$,
    \begin{align*}
        \left|-s\frac{\sigma^2}{2}|\xi|^2+(1-s)R(\xi)\right|&\le\phi\frac{\sigma^2}{2}|\xi|^2+C(1+\phi)\frac{\sigma^2}{2}|\xi|^{2+\varepsilon}\\
        &\le\left(\phi+C(1+\phi)r^\varepsilon\right)\frac{\sigma^2}{2}|\xi|^2\le\frac{c_\eta}{2}\frac{\sigma^2}{2}|\xi|^2,
    \end{align*}
    so long as we take $\phi$ and $r$ sufficiently small. Therefore by \eqref{ineq:lower bound for |s+q|},
    \begin{align}
    \label{ineq:bound of 1-z phi(xi) for small xi}
        |1-z\varphi(\xi)|\ge\left|s+\frac{\sigma^2}{2}|\xi|^2\right|-\left|s\frac{\sigma^2}{2}|\xi|^2-(1-s)R(\xi)\right|\ge\frac{c_\eta}{2}\left(|1-z|+\frac{\sigma^2}{2}|\xi|^2\right).
    \end{align}
    Combining \eqref{ineq:bound of 1-z phi(xi) for large xi} and \eqref{ineq:bound of 1-z phi(xi) for small xi}, we see that for any compact subset $K\subset\Delta_{\phi,\eta}$, there is $c_K>0$ such that uniformly in $z\in K$, $\xi\in\TT^2$, $|1-z\phi(\xi)|\ge c_K$. In particular, by dominated convergence, we can extend $z\mapsto G_z(x)$ analytically to $\Delta_{\phi,\eta}$. 

    Fix $x=0$, $r$ sufficiently small and $s=1-z$, then write
    \begin{equation*}
        G_{1-s}(0)=I_r^{1}(s)+I_r^{2}(s),
    \end{equation*}
    where 
    \begin{equation*}
        I_r^{1}(s)=\frac{1}{(2\pi)^2}\int_{|\xi|<r}\frac{\d\xi}{s+\frac{\sigma^2}{2}|\xi|^2},
    \end{equation*}
    and
    \begin{align*}
        I_r^{2}(s)&=\frac{1}{(2\pi)^2}\int_{|\xi|<r}\d\xi\left(\frac{1}{1-(1-s)\varphi(\xi)}-\frac{1}{s+\frac{\sigma^2}{2}|\xi|^2}\right)+\frac{1}{(2\pi)^2}\int_{|\xi|\ge r}\frac{\d\xi}{1-(1-s)\varphi(\xi)}\\
       & =: J_r^1(s)+J_r^2(s).
    \end{align*}
    Thanks to \eqref{eq:asymptotic expansion of characteristic function}, $J_r^1(s)$ is well defined at $s=0$. In polar coordinates, we have 
    \begin{align}
    I_r^1(s)&=\frac{1}{2\pi}\int_0^r\frac{u}{s+\frac{\sigma^2}{2}u^2}\d u=\frac{1}{2\pi\sigma^2}\left(\mathrm{Log}\left(s+\frac{\sigma^2r^2}{2}\right)-\mathrm{Log}(s)\right)\nonumber\\
    &=\lambda\mathrm{Log}(1/s)+\lambda\log\left(\frac{\sigma^2r^2}{2}\right)+O(|s|)=\lambda L(z)+\lambda\log\left(\frac{\sigma^2r^2}{2}\right)+O(|1-z|),\label{eq:expansion of I_r^1(s)}
    \end{align}
    where $\mathrm{Log}$ is the principle branch of the complex logarithm.
    
    To study $J_r^{1}(s)$, we control $J_r^1(s)-J_r^1(0)$ for $s<r$. We introduce the shorthand notations $q:=(\sigma^2/2)|\xi|^2$ and $p:=1-\varphi(\xi)-q$ and rewrite the integrand in $J_r^1(s)$ as $f(s,q)+g(s,q)$, where
    \begin{equation*}
        f(s,q):=\frac{1}{s+(1-s)(q+p)}-\frac{1}{s+(1-s)q},
    \end{equation*}
    and
    \begin{equation*}
        g(s,q):=\frac{1}{s+(1-s)q}-\frac{1}{s+q}.
    \end{equation*}
    Taking $s$ sufficiently small and applying \eqref{ineq:lower bound for |s+q|}, we have on one hand 
    \begin{equation}
    \label{ineq:bound for g(s,q)}
        |g(s,q)|=\frac{|s| q}{|s+(1-s)q||s+q|}\lesssim\frac{|s| q}{(|s|+q)^2}.
    \end{equation}
    Note that the integral of the term on the right hand side is $O(|s|\log(1/|s|))$. Reasoning similarly for $f(s,q)$ and also invoking \eqref{eq:asymptotic expansion of characteristic function}, we obtain
    \begin{equation*}
        |f(s,q)|\lesssim\frac{q^{1+\varepsilon/2}}{(|s|+q)^2}.
    \end{equation*}
    In particular, the integral of $f(s,q)-f(0,q)$ over $q\in(0,|s|)$ is $O(|s|^{\varepsilon/2})$. To handle the integral over $q\ge|s|$, we write 
    \begin{equation*}
        f(s,q)-f(0,q)=s\int_0^1 \partial_w f(ws,q)\d w.
    \end{equation*}
    A direct calculation yields $|\partial_w f(ws,q)|\lesssim q^{-2+\varepsilon/2}$. In particular,
    \begin{equation}
    \label{ineq:bound for f(s,q)}
        \int_{|s|}^r|f(s,q)-f(0,q)|\d q\lesssim |s|\int_{|s|}^r\frac{\d q}{q^{2-\varepsilon/2}}=O(|s|^{\varepsilon/2}).
    \end{equation}
    Finally, noting that for $\varepsilon\in(0,1)$, $|s|\log(1/|s|)=O(|s|^{\varepsilon/2})$ and combining \eqref{ineq:bound for g(s,q)} and \eqref{ineq:bound for f(s,q)}, we get
    \begin{equation*}
        J_r^1(s)=J_r^1(0)+O(|s|^{\varepsilon/2}).
    \end{equation*}
    In particular, by letting $\kappa_S:=J_r^{1}(0)+\lambda\log\left(\sigma^2r^2/2\right)$ (which doesn't depend on the choice of $r$, see Lemma \ref{lemma:kappa_S is well defined}) and recalling \eqref{eq:expansion of I_r^1(s)}, we have successfully shown that
    \begin{equation*}
        G_z(0)=\lambda L(z)+\kappa_S+O(|1-z|^{\varepsilon/2}).
    \end{equation*}
    To estimate $G_z(x)$, we note that by Fourier inversion,
    \begin{equation*}
        a(x)=\frac{1}{(2\pi)^2}\int_{\TT^2}\d\xi\frac{1-\exp(\i\langle x,\xi\rangle}{1-\varphi(\xi)},
    \end{equation*}
    whence 
    \begin{equation*}
        G_{1-s}(x)-G_{1-s}(0)+a(x)=-\frac{s}{(2\pi)^2}\int_{\TT^2}\d\xi \frac{\varphi(\xi)(\exp(\i\langle x,\xi\rangle)-1)}{(1-(1-s)\varphi(\xi))(1-\varphi(\xi))}.
    \end{equation*}
    For $|\xi|>1$, strong aperiodicity implies that the integral is $O(1)$. For $|\xi|\le 1$, we have on one hand $|\exp(\i\langle x,\xi\rangle)-1|\lesssim|\xi|$, $|1-\varphi(\xi)|\ge C|\xi|^2$ for some $C>0$ and thanks to \eqref{ineq:lower bound for |s+q|}, for some $\tilde C>0$,
    \begin{equation*}
        |1-(1-s)\varphi(\xi)|\ge\tilde C(|s|+|\xi|^2).
    \end{equation*}
    After passing to polar coordinates, the contribution is thus bounded up to some constant by 
    \begin{equation*}
        |s|\int_{0}^1\frac{\d u}{|s|+u^2}=O(|s|^{1/2}).
    \end{equation*}
    All in all, this gives
    \begin{align*}
        G_z(x)&=G_z(0)-a(x)+O(|1-z|^{1/2})\\
        &=\frac{L(z)}{2\pi\sigma^2}+\kappa_S-a(x)+O(|1-z|^{\varepsilon/2}),
    \end{align*}
    since $\varepsilon<1$.

\subsubsection{Proof of Proposition \ref{prop:G_A is invertible and with holomorphic inverse}}
\label{appendix:proof of G_A is invertible and with holomorphic inverse}

    We start by showing that $G_{z,A}$ is invertible for all $z\in\mathbb D\backslash\{1\}$. Note that for any such $z$, it suffices to show that for any $v\in\CC^{A}\backslash\{0\}$, $\mathfrak R(v^*G_{z,A}v)>0$. Using Fourier inversion to write
    \begin{equation*}
        G_{z,A}(x,y)=\frac{1}{(2\pi)^2}\int_{\TT^2}\d \xi\frac{\exp(\i\langle\xi,x-y\rangle)}{1-z\varphi(\xi)},
    \end{equation*}
    and defining $P_v(\xi):=\sum_{a\in A}v_a\exp(-\i\langle\xi,a\rangle)$, then a direct calculation shows that
    \begin{equation*}
        v^* G_{z,A}v=\frac{1}{(2\pi)^2}\int_{\TT^2}\d\xi\frac{|P_v(\xi)|^2}{1-z\varphi(\xi)},
    \end{equation*}
    and so taking the real part,
    \begin{equation*}
        \mathfrak R(v^*G_{z,A}v)=\frac{1}{(2\pi)^2}\int_{\TT^2}\d\xi|P_v(\xi)|^2\frac{1-\mathfrak R(z\varphi(\xi))}{|1-z\varphi(\xi)|^2}.
    \end{equation*}
    Clearly, since $|z|\le1$ and $S$ is strongly aperiodic which implies that $|\varphi(\xi)|=1$ if and only if $\xi=0$, then $\mathfrak R(z\varphi(\xi))=1$ only if $z=1$, which is excluded. Since $P_v$ isn't identically $0$, we obtain $\mathfrak R(v^*G_{z,A}v)>0$, whence $G_{z,A}$ is invertible for $z\in\mathbb D\backslash\{1\}$. By letting $K_\varepsilon:=\{z\in\mathbb D\mid |1-z|>\varepsilon/2\}$, then the same conclusion holds on $K_\varepsilon$ for any $\varepsilon\in(0,1/2)$. By continuity of the determinant and of the mapping $z\mapsto G_{z,A}$ which follows from the proof of Proposition \ref{prop:expansion of G_z(x)}, since $K_\varepsilon$ is compact we have $\inf_{z\in K_{\varepsilon}}|\det(G_{z,A})|>0$. In particular, for some $\phi>0$ sufficiently small, we also have $\inf_{z\in K_{\varepsilon}^\phi}|\det(G_{z,A})|>0$, where $K_\varepsilon^\phi$ is a $\phi$-neighbourhood of $K_\varepsilon$. Therefore, $G_{z,A}$ is invertible on $K_{\varepsilon}^\phi$ for $\phi$ sufficiently small.

    Define the domain $\Delta_{\phi,\eta}^\varepsilon:=\{z\in\Delta_{\phi,\eta}\mid |1-z|<\varepsilon\}$ for $\varepsilon\in(0,\phi/2)$. Once we show that $G_{z,A}$ is invertible for $z\in\Delta_{\phi,\eta}^\varepsilon$, we will be able to conclude, since $K_{\varepsilon}^\phi\cup\Delta_{\phi,\eta}^\varepsilon$ contains $\Delta_{\phi,\eta}$. To show this is the case, we apply Proposition \ref{prop:expansion of G_z(x)} to write 
    \begin{equation}
    \label{eq:expression of G_(z,A)}
        G_{z,A}=(\lambda L(z)+\kappa_S)J-\mathcal A_A+E_{z,A},
    \end{equation}where $\mathcal A_A:=(a(x-y))_{x,y\in A}$, $\lVert E_{z,A}\rVert\le C|1-z|^{\beta/2}$ for some $C>0$ and we recall that $J=\mathbf 1\mathbf 1^T$. We write $M_{z,A}:=(\lambda L(z)+\kappa_S)J-\mathcal A_A$. If we knew that for the right $\phi,\varepsilon>0$ we had
    \begin{equation}
        \label{statement:M_(z,A) invertible}
        \text{$M_{z,A}$ is invertible for every $z\in\Delta_{\phi,\eta}^\varepsilon$,}
    \end{equation}
    and
    \begin{equation}
        \label{statement:norm of inverse of M_(z,A) uniformly bounded}
        S:=\sup_{z\in\Delta_{\phi,\eta}^\varepsilon}\lVert M_{z,A}^{-1}\rVert<\infty,
    \end{equation}
    then by \eqref{eq:expression of G_(z,A)}, we would have 
    \begin{equation}
    \label{eq:factorized version of G_(z,A)}
        G_{z,A}=M_{z,A}(\mathrm{Id}+M_{z,A}^{-1}E_{z,A}),
    \end{equation}
    where for any $z\in\Delta_{\phi,\eta}^\varepsilon$, taking $\varepsilon$ small enough so that $CS\varepsilon^{\beta/2}<1/2$,
    \begin{equation*}
        \lVert M_{z,A}^{-1}E_{z,A}\rVert\le CS|1-z|^{\beta/2}<1/2.
    \end{equation*}
    This would yield the invertibility of $\mathrm{Id}+M_{z,A}^{-1}E_{z,A}$ for $z\in\Delta_{\phi,\eta}^\varepsilon$, and thus $G_{z,A}$, which would conclude the proof. Therefore, it remains to show \eqref{statement:M_(z,A) invertible} and \eqref{statement:norm of inverse of M_(z,A) uniformly bounded}.

    We start with \eqref{statement:M_(z,A) invertible}. Define the hyperplane 
    \begin{equation*}
        H:=\{v\in\CC^A\mid \mathbf 1^Tv=0\}.
    \end{equation*}
    We show that if $v\in H$ and $\mathcal A_A v\in\mathrm{span}(\mathbf 1)$, then $v=0$. By considering the real and imaginary parts, it is sufficient to show that this holds for $v\in H\cap\RR^{A}$. Using Fourier inversion, we have 
    \begin{equation*}
        a(x)=\frac{1}{(2\pi)^2}\int_{\TT^2}\d\xi\frac{1-\exp(\i\langle\xi,x\rangle)}{1-\varphi(\xi)}.
    \end{equation*}
    In particular, reasoning as we did in the first part of the proof, supposing further that $v\ne 0$,
    \begin{align*}
        v^T\mathcal A_A v&=-\frac{1}{(2\pi)^2}\int_{\TT^2}\d\xi\frac{|P_v(\xi)|^2}{1-\varphi(\xi)}\\
        &=-\frac{1}{(2\pi)^2}\int_{\TT^2}\d\xi|P_v(\xi)|^2\mathfrak R\left(\frac{1}{1-\varphi(\xi)}\right)\\
        &=-\frac{1}{(2\pi)^2}\int_{\TT^2}\d\xi|P_v(\xi)|^2\frac{1-\mathfrak R(\varphi(\xi))}{|1-\varphi(\xi)|^2}<0.
    \end{align*}
    Therefore, if $\mathcal A_A v\in\mathrm{span}(\mathbf 1)$, then for some $\mu\in\CC$, $v^T\mathcal A_A v=\mu v^T\mathbf 1=0$, which implies $v=0$. 

    Recall the harmonic measure of $A$, noted $\mu_A$, satisfies $\mathcal A_A\mu_A=\Cap(A)\mathbf 1$. Let $v\in \ker(M_{z,A})$, and define $\alpha:=\mathbf 1^Tv$, $w=v-\alpha\mu_A$. since $\mu_A$ is a probability measure, we therefore have $w\in H$. We now calculate
    \begin{align*}
        M_{z,A}v&=(\lambda L(z)+\kappa_S) Jv-\mathcal A_Av\\
        &=\alpha(\lambda L(z)+\kappa_S)\mathbf 1-\mathcal A_A(w+\alpha\mu_A)\\
        &=\alpha(\lambda L(z)+\kappa_S-\Cap(A))\mathbf 1-\mathcal A_Aw=0,
    \end{align*}
    whence $\mathcal A_Aw=\alpha(\lambda L(z)+\kappa_S-\Cap(A))\mathbf 1$. By what we just saw, this implies $w=0$, and therefore $v=\alpha\mu_A$ and also $\alpha(\lambda L(z)+\kappa_S-\Cap(A))=0$. Taking $\varepsilon>0$ sufficiently small and $z\in\Delta_{\phi,\eta}^\varepsilon$ guarantees that $\lambda L(z)+\kappa_S-\Cap(A)\ne0$, and so for such $z$ we must have $\alpha=0$, whence $v=\alpha\mu_A=0$, which proves \eqref{statement:M_(z,A) invertible}.

    Let $\pi$ be the projection onto $H$ along $\mathrm{span}(\mathbf 1)$. Define the mapping $f$ from $H$ to $H$ by $f(w)=-\pi(\mathcal A_Aw)$. If $w\in\ker(f)$, then $\pi(\mathcal A_A w)=0$, whence $\mathcal A_A w=c\mathbf 1$ for some $c\in\CC$, so $w=0$ by what we saw previously, which proves that $f$ is injective and thus an isomorphism. In particular, there is $C_f>0$ such that for any $w\in H$, $\lVert f^{-1}(w)\rVert\le C_f\lVert w\rVert$. Now, fix $y\in\CC^{A}$. Let $v\in\CC^{A}$, $\alpha=\mathbf 1^Tv$ and $w=v-\alpha\mu_A$ as before. Suppose that $M_{z,A}v=y$. Then by the same calculation as we performed previously,
    \begin{equation*}
        y=\alpha(\lambda L(z)+\kappa_S-\Cap(A))\mathbf 1-\mathcal A_Aw=:\tilde\alpha\mathbf 1-\mathcal A_A w,
    \end{equation*}
    and applying $\pi$ to both sides we get
    \begin{equation*}
        \pi(y)=-\pi(\mathcal A_Aw)=f(w),
    \end{equation*}
    whence $w=f^{-1}(\pi(y))$ and thus $\lVert w\rVert\le C_f\lVert y\rVert$. Furthermore, since $\tilde\alpha\mathbf 1=y+\mathcal A_Aw$, then this implies that $|\tilde\alpha|\le C'\lVert y\rVert$ for some $C'>0$. In particular,
    \begin{equation*}
        |\alpha|\le\frac{C'\lVert y\rVert}{|\lambda L(z)+\kappa_S-\Cap(A)|}.
    \end{equation*}
    Taking $\varepsilon$ sufficiently small so that the denominator is greater than $1$ for any $z\in\Delta_{\phi,\eta}^\varepsilon$, we get $|\alpha|\le C'\lVert y\rVert$, whence for some $C''>0$,
    \begin{align*}
        \lVert M_{z,A}^{-1} y\rVert=\lVert v\rVert\le|\alpha |\lVert\mu_A\rVert+\lVert w\rVert\le C''\lVert y\rVert,
    \end{align*}
    which concludes the proof of \eqref{statement:norm of inverse of M_(z,A) uniformly bounded}.

\subsubsection{Proof of Corollary \ref{corollaire:dev asymp de 1^T G_(z,A) 1}}
\label{appendix:proof of dev asymp de 1^T G_(z,A) 1}

    Note that $M_{z,A}\mu_A=(\lambda L(z)+\kappa_S-\Cap(A))\mathbf 1$, and so 
    \begin{equation*}
        M_{z,A}^{-1}\mathbf 1=\frac{\mu_A}{\lambda L(z)+\kappa_S-\Cap(A)}.
    \end{equation*}
    Combining this with \eqref{eq:factorized version of G_(z,A)} yields
    \begin{align*}
        \mathbf 1^TG_{z,A}^{-1}\mathbf 1&=\mathbf 1^T\left(\sum_{n\ge0}(-1)^n\left(M_{z,A}^{-1}E_{z,A}\right)^n\right)\frac{\mu_A}{\lambda L(z)+\kappa_S-\Cap(A)}=\mathbf 1^T\frac{\mu_A+O(|1-z|^{\varepsilon/2})}{\lambda L(z)+\kappa_S-\Cap(A)}\\
        &=\frac{1}{\lambda L(z)+\kappa_S-\Cap(A)}+O\left(\frac{|1-z|^{\varepsilon/2}}{L(z)}\right)\\
        &=\sum_{k=1}^N\frac{\left(\Cap(A)-\kappa_S\right)^{k-1}}{\lambda^k L(z)^k}+O\left(\frac{1}{L(z)^{N+1}}\right),
    \end{align*}
where we used the Taylor series expansion of order $N$ of $z\mapsto(\lambda L(z)+\kappa_S-\Cap(A))^{-1}$ in the last line.
\subsection{The constant $\kappa_S$}
\label{appendix:kappa_S}
The aim of this section is to relate the constant $\kappa_S$ to the truncated Green's function $h(n):=\sum_{k=0}^n\PP(S_k=0)$, which will allow us to compare Theorem \ref{theorem:expansion of the mean} to the expansion of the range obtained by Le Gall and Rosen (\cite{LeGallRosen1991}, Theorem $6.9$). We start by showing that the constant $\kappa_S$ is well defined.

\begin{lemma}
    \label{lemma:kappa_S is well defined}
    The mapping $\kappa_S:r\mapsto J_r^1(0)+\lambda\log(\sigma^2 r^2/2)$ is constant on $(0,\pi)$.
\end{lemma}

\begin{proof}
    Let $0<r_1<r_2<\pi$. Then using polar coordinates,
    \begin{align*}
        \kappa_S(r_2)-\kappa_S(r_1)&=-\frac{2/\sigma^2}{(2\pi)^2}\int_{r_1<|\xi|<r_2}\frac{\d\xi}{|\xi|^2}+\frac{1}{2\pi\sigma^2}\log\left(\frac{r_2^2}{r_1^2}\right)\\
        &=-\frac{1}{\pi\sigma^2}\int_{r_1}^{r_2}\frac{\d u}{u}+\frac{1}{\pi\sigma^2}\log\left(\frac{r_2}{r_1}\right)=0.\qedhere
    \end{align*}
\end{proof}
According to Proposition \ref{prop:expansion of G_z(x)}, and recalling the notation $\lambda=1/2\pi\sigma^2$, $\kappa_S$ can be defined as
\begin{equation*}
    \kappa_S=\lim_{\substack{z\rightarrow1 \\ z<1}}\left(G_z(0)-\lambda\log\left(\frac{1}{1-z}\right)\right).
\end{equation*}
Therefore, it matches the constant $c_X$ of Bass and Rosen (see \cite{BassRosen}, Lemma $A.1$).

Under the finite $2+\delta$ moment assumption, we have the following strengthening of the local limit theorem,
\begin{equation}
\label{eq:strengthened local limit theorem}
    \PP(S_n=0)=\frac{\lambda}{n}+O(n^{-1-\varepsilon/2}),
\end{equation}
for any $\varepsilon\in(0,1\wedge\delta)$. Therefore, we immediately get the asymptotic expansion
\begin{equation}
\label{eq:asymptotic expansion of truncated green's function}
    h(n)=\lambda\log n+c_S+o(1),
\end{equation}
where $c_S:=1+\lambda C_E+\sum_{k\ge1}\left(\PP(S_k=0)-\lambda/k\right)$ and where $C_E$ is the Euler Mascheroni constant. Note that, as $\kappa_S$, $c_S$ is a constant that depends only on the distribution of $S$. We show the following link between $c_S$ and $\kappa_S$.
\begin{lemma}
    \begin{equation*}
        \kappa_S=c_S-\lambda C_E.
    \end{equation*}
\end{lemma}
\begin{proof}
    Since the Taylor series of $\log(1/(1-z))$ is $\sum_{k\ge1} z^k/k$, we simply rewrite
    \begin{equation*}
        G_z(0)-\lambda\log\left(\frac{1}{1-z}\right)=1+\sum_{k\ge1}z^k\left(\PP(S_k=0)-\frac{\lambda}{k}\right),
    \end{equation*}
    and thanks to \eqref{eq:strengthened local limit theorem} and dominated convergence, the right hand side converges to $c_S-\lambda C_E$ as $z\rightarrow1$, which concludes.
\end{proof}
According to \cite{LeGallRosen1991}, Theorem $6.9$, we have 
\begin{equation*}
    \EE[\R_n]=\frac{n}{h(n)}+\lambda\frac{n}{h(n)^2}+o\left(\frac{n}{h(n)^2}\right).
\end{equation*}
Applying \eqref{eq:asymptotic expansion of truncated green's function} yields
\begin{align*}
    \EE\left[\R_n\right]&=\frac{n}{\lambda\log n+c_S+o(1)}+\lambda\frac{n}{(\lambda\log n+c_S+o(1))^2}+o\left(\frac{n}{(\log n)^2}\right)\\
    &=\lambda^{-1}\frac{n}{\log n}+\lambda^{-1}\left(1-\lambda^{-1}c_S\right)\frac{n}{(\log n)^2}+o\left(\frac{n}{(\log n)^2}\right)\\
    &=2\pi\sigma^2\frac{n}{\log n}+2\pi\sigma^2\left(1-C_E-2\pi\sigma^2\kappa_S\right)\frac{n}{(\log n)^2}+o\left(\frac{n}{(\log n)^2}\right),
\end{align*}
which agrees with Theorem \ref{theorem:expansion of the mean} in the case $A=\{0\}$, $N=2$.

\paragraph{Acknowledgements.}  We thank Amine Asselah for suggesting the problem to us and for careful proofreading of this work, as well as Quentin Berger for useful comments about the presentation of the work. M.B. is grateful to the CNRS project MITI and to the Chair “Modélisation Mathématique et Biodiversité of VEOLIA-Ecole Polytechnique-MNHN-F.X., by the European Union (ERC, SINGER, 101054787) for funding his work. 

\paragraph{AI statement.} This work is the fruit of human collaboration. All the results and proofs presented here have been found and written down without any use of AI.

\printbibliography[title={Bibliography}]

\end{document}